\documentclass{article}
\usepackage{cmap} % for russuan cope paste

\usepackage[T2A]{fontenc}

\usepackage[utf8]{inputenc}
\usepackage[english]{babel}
\usepackage[tbtags]{amsmath}
\usepackage{amsfonts,amssymb,amscd}
\usepackage{mathrsfs}
\usepackage{graphics}
\usepackage{tikz-cd}

\usepackage[pdftex,unicode]{hyperref}

\let\No=\textnumero

\usepackage[hyper]{izv2e}
\JournalName{MATHEMATICAL SERIES}

\def\term#1{{\em #1}}

  \def\R{\mathbb R}

\def\N{\mathbb N}

\def\om{\omega}

\def\supp{\operatorname{supp}}

\def\be{\beta}

\def\al{\alpha}
\def\ga{\gamma}
\def\la{\lambda}
\def\ph{\varphi}
\def\de{\delta}
\def\ka{\kappa}

\def\wt#1{\widetilde{#1}}

\def\cQ{\mathcal{Q}}
\def\cC{\mathcal{C}}
\def\cR{\mathcal{R}}
\def\cS{\mathcal{S}}

\def\cV{\mathcal V}
\def\cM{\mathcal M}
\def\cP{\mathcal P}
\def\cU{\mathcal U}
\def\cA{\mathcal A}
\def\cE{\mathcal{E}}

\def\cN{\mathcal{N}}
\def\cL{\mathcal{L}}

\def\gM{{\mathfrak{M}}}

\def\cl#1{\overline{#1}}
\def\clx#1#2{\overline{#2}^{\,#1}}

\def\id{\operatorname{id}}

\def\bt{{\beta}}
\def\bt{{\beta}\mskip 1}
\def\bt{{\beta}\mskip 0.4mu}

\def\oms{\om^*}

\def\es{\varnothing}

\def\sset#1{\{#1\}}

\def\set#1{\bbset#1\eeset}
\def\bbset#1:#2\eeset{\{#1\,:\,#2\}}

\def\bbsett#1:#2\eesett{\{#1\,:\,\text{#2}\}}

\def\ibbset#1:#2\ieeset{(#1)_{#2}}

\newcommand\restrB[2]{\ensuremath{\left.#1\right|_{#2}}}

\def\restr#1#2{\restrB{#1}{#2}}

\def\oo#1/{$O_{#1}$}

\def\ep{\varepsilon}

\def\gd/{$G_\delta$}

\def\rarr{\Rightarrow}
\def\larr{\Leftarrow}

\def\DD{\operatornamewithlimits{%
  \mathchoice{\vcenter{\hbox{\Large $\bigtriangleup$}}}
             {\mathrm{\bigtriangleup}}
             {\mathrm{\bigtriangleup}}
             {\mathrm{\bigtriangleup}}}}

\def\diag{\DD}

\def\si{\sigma}

\def\vh#1{\hat{#1}}

\def\vt#1{\tilde{#1}}

\def\cc#1{\cC_{#1}}
\def\ccc#1{$(\cc{#1})$}

\def\prt#1{{\mathfrak{P}}(#1)}

\def\cccc/{$ccc$}

\def\oms{\om^*}
\def\Si{\Sigma}

\def\nm#1{\mathsf{\mathbf{#1}}}
\def\nb{\nm{2}}

\def\ct{\nb^{<\om}}
\def\cs{\nb^{\om}}

\def\rst{\restriction}

\def\cms#1{$(\mathsf{MS}_{#1})$}

\def\:{\colon}

\def\see\cite#1{(see~\cite{#1})}
\def\seee\cite[#1]#2{(see~\cite[#1]{#2})}

\newtheorem{theorem}{Theorem}[section]
\newtheorem{lemma}[theorem]{Lemma}
\newtheorem{problem}[theorem]{Problem}
\newtheorem{proposition}[theorem]{Proposition}
\newtheorem{assertion}[theorem]{Proposition}
\newtheorem{cor}[theorem]{Corollary}

\theoremstyle{definition}
\newtheorem{proof}{Proof}

\newtheorem{example}[theorem]{Example}
\newtheorem{note}[theorem]{Note}

\begin{document}
%\def\refname{Литература}

%\def\path#1{ #1}

%\rusisspages{15--39}

%\frenchspacing

\title{Stone--\v Cech Compactifications of Spaces of Probability Measures}

\author{E.\,A.~Reznichenko}

\address{Lomonosov Moscow State University, Moscow}
\email{erezn@inbox.ru}

\date{December 31, 2024}\udk{515.122.4 515.122.252 515.122.536 519.2}
%  
%\datedor{26 ноября 2023 г.}
%\dateprin{26 ноября 2023 г.}

%\doi{10.4213/faa9999}

%\dedicatory{Анатолию Моисеевичу}

\markboth{Compactifications of spaces of probability measures}{E.\,A.~Reznichenko}

\maketitle

\begin{fulltext}

%\thanks{Работа выполнена при поддержке.}

%\subjclass{??????}
% 28A33 46E27 60B05 54C45 54C25

\begin{abstract}
It is proved that $P(\bt X)=\bt P(X)$ if and only if $P(X)$ is a pseudocompact space, 
where $P(X)$ is the space of Radon probability measures with the weak topology and $\bt X$ is the Stone--\v Cech 
compactification of $X$. A locally compact pseudocompact space $X$ is constructed such that $P(X)$ is not 
pseudocompact. Conditions are obtained under which $P(X)$ is pseudocompact. 
\end{abstract}


\section{Introduction}\label{sec:intro}
%%% begin file('tex/sec/intro.tex') %%%
In many problems of measure theory and functional analysis, when considering measures on a topological
space $X$, it is very useful to use the Stone--\v Cech compactification $\beta X$ of this space.
Moreover, the space $P(X)$ of Radon probability measures on $X$ with the weak topology is naturally embedded
in the compact space $P(\beta X)$ of measures on the compactification, which also yields a continuous mapping 
from the Stone--\v Cech compactification $\beta P(X)$ of $P(X)$ to $P(\beta X)$. The question naturally arises of 
whether this mapping is a homeomorphism (in this case, $P(\beta X)$ and $\beta P(X)$ are considered to be the 
same). This question was first studied in the papers \cite{Bogachev2024-1ru} and \cite{Bogachev2024dan-ru}, 
where it was noted that $P(\beta X)$ and $\beta P(X)$ do not always coincide. Moreover, it they do, then 
the spaces $X$ and $P(X)$ are pseudocompact, so that in the class of metrizable spaces the coincidence occurs only for 
compact spaces. Furthermore, the cited papers contain examples of non-compact pseudocompact spaces $X$ for which 
the coincidence occurs and a number of open questions, including that of whether the equality 
$P(\beta X)=\beta P(X)$ is equivalent to the pseudocompactness of $P(X)$ or $X$. In this paper, we continue to 
study the question of when the Stone--\v Cech compactification of $P(X)$ coincides with the space $P(\beta X)$ 
and answer some of the questions posed in \cite{Bogachev2024-1ru} and \cite{Bogachev2024dan-ru}. In particular, 
we prove that if $P(X)$ is pseudocompact, then $P(\beta X)$ and $\beta P(X)$ coincide (Theorem~\ref{t:intro:3}) 
and construct a locally compact pseudocompact space $X$ for which $P(\beta X)\neq\beta P(X)$ 
(Example~\ref{e:intro:example:3}).

In \cite{BanakhChigogidzeFedorchuk2003}, it was proved that the space $P_\si(X)$ of $\si$-additive probability
measures on the Baire $\si$-algebra of $X$ is realcompact and the following question was posed: Is it 
true that the Hewitt extension $\nu P(X)$ of $P(X)$ coincides with $P_\si(X)$? If $X$ is pseudocompact, then 
$P_\si(X)=P(\bt X)$ and $\nu P(X)=P(\bt X)$ imply that so is of $P(X)$. Therefore, 
Example~\ref{e:intro:example:3} also give a negative answer to Problem~5.5 
in~\cite{BanakhChigogidzeFedorchuk2003}.

The problem of describing spaces $X$ for which $P(\beta X)=\beta P(X)$ is a special case of the following general 
problem.

\begin{problem}\label{pq:intro:intro:1}
Let $X$ be a space with some structure. When can this structure be extended to~$\be X$?
\end{problem}

Some of the known results in this direction are Gliksberg's theorem \see\cite{gli1959} that,  
for a space $X$, the equality $\bt (X^2)= (\bt X)^2$ holds if and only if $X^2$ is pseudocompact  
and the Comfort--Ross theorem \see\cite{ComfortRoss1966} that if $G$ is a topological group,
then the continuous group operations on $G$ can be extended to $\bt G$ if and only if $G$ is pseudocompact.
Moreover, if $\bt G$ is homeomorphic to some topological group, then $G$ is pseudocompact.

In \cite{ReznichenkoUspenskij1998}, the Comfort--Ross theorem was extended to Mal'tsev spaces.
A Mal'tsev space is a space $X$ with a continuous Mal'tsev operation, that is, a ternary operation
$m\: X^3\to X$ for which the identities $m(x,y,y)=m(y,y,x)=x$ hold.
For a Mal'tsev space $X$, the following conditions are equivalent:
(1)~$X$ is pseudocompact; (2)~the Mal'tsev operation on $X$ extends to a continuous operation on $\bt X$;
(3)~$\bt X$ is homeomorphic to some compact Mal'tsev space \see\cite{ReznichenkoUspenskij1998}.

There is a standard Mal'tsev operation on groups: $m(x,y,z)=xy^{-1}z$. Mal'tsev operation plays an important 
role in the study of universal algebras. Its fundamental importance in topological 
algebra is determined by the fact that quotient homomorphisms of topological universal algebras with a Mal'tsev 
operation are open \see\cite{Malcev1954ru}. This ensures that compact Mal'tsev spaces are Dugundji   
\see\cite{Uspenskii1989ru}. Dugundji compact spaces have the following important property: if $X$ is a Dugundji 
compact space and $Y\subset X$ is dense in $X$, then the following conditions are equivalent: (1)~$Y$ is 
pseudocompact; (2)~$\bt Y=X$; (3)~$\bt Y$ is a Dugundji compact space 
\see\cite{Uspenskii1989ru,ReznichenkoUspenskij1998}.

The problem of extending operations from $X$ to $\bt X$ in the general case of a topological universal algebra 
$X$ was considered in \cite{Reznichenko2024efualg}. A universal algebra is a set with a collection of $n$-ary 
operations. Examples of universal algebras include groups, semigroups, and sets with a Mal'tsev operation. 
Vector spaces are universal algebras as well; in addition to the additive group operations on a vector 
space, there are also continuum many unary operations $v\mapsto \la v$ for $\la\in\R$. A convex set $C$ can also 
be viewed as a universal algebra with continuum many binary operations $(x,y)\mapsto \la x +(1-\la)y$ for 
$\la\in [0,1]$. Using results of \cite{Reznichenko2024efualg}, the author proved in 
\cite{Reznichenko2024betacs-ru} that, given a pseudocompact convex space $C$, the continuous operations on 
$C$ extend to separately continuous operations on $\bt C$. In particular, $\bt C$ is a path-connected space.

Convex compact spaces and spaces of probability measures on compact spaces are not necessarily Dugundji; thus,
there is no complete analogy with the theory of topological groups. There exists a compact space $X$ 
such that $P(X)$ contains a dense pseudocompact convex subset $Y$ 
for which $\bt Y\neq P(X)$ (Example~\ref{e:intro:conv:1}).

However, there is an analogy with other related results. Namely, for a space $X$, the following conditions are 
equivalent: (1)~$P(X)$ is pseudocompact; (2)~$\bt P(X) = P(\bt X)$; (3)~$\bt P(X)$ is homeomorphic to a 
compact convex set (\cite{Bogachev2024dan-ru,Bogachev2024-1ru}, Theorem~\ref{t:intro:3}).

\begin{problem}\label{pq:intro:intro:2}
Determine for which (normal) functors $F$ from the category of Tychonoff spaces to the category of 
Tychonoff spaces the following conditions are equivalent: (1)~$F(X)$ is pseudocompact; (2)~$\be F(X)=F(\be X)$; 
(3)~$\be F(X)$ is homeomorphic to $F(Y)$ for some compact space~$Y$. 
\end{problem}

\begin{problem}\label{pq:intro:intro:3}
Let $X$ be a space such that $\bt(X^2)$ is homeomorphic to $Y^2$ for some space $Y$. Is $X^2$ pseudocompact?
\end{problem}

\subsection{Notation, Definitions, and Conventions}

%%% begin file('tex/sec/defs.tex') %%%

Let $\om=\sset 0 \cup \N$ and $\om_1$ be the first countable and the first uncounable ordinal, respectively. 
By space we always mean a Tychonoff space.

The space $P(X)$ is a subspace of the compact space $P(\bt X)$; hence the embedding $i$
of $P(X)$ into $P(\bt X)$ extends to a continuous mapping $\hat i\: \bt P(X)\to P(\bt X)$.
By the question of the equality $\bt P(X)=P(\bt X)$ we mean the question of whether $\hat i$ is a homeomorphism.

Let $Y\subset X$. A set $Y\subset X$ is $C$-embedded ($C^*$-embedded) in a space $X$ if
every (bounded) continuous function on $Y$ extends to a continuous function on $X$.
The Stone--\v Cech compactification $\bt X$ of a space $X$ is characterized by 
that $\bt X$ is a compact extension of $X$ and $X$ is $C^*$-embedded in $\bt X$.
The equality $\bt P(X)=P(\bt X)$ means that $P(X)$ is $C^*$-embedded in $P(\bt X)$.

We denote by $\diag$ the \term{diagonal product} of mappings.
Given a set $X$  and mappings $f\: X\to Y$ and $g\: X\to Z$,  $(f\diag g)(x)=(f(x),g(x))$.
Given  be a family of sets $Y_\al$, $\al\in A$ and a family of mappings $f_\al\: X\to Y_\al$, $\al\in A$, 
the mapping
\[
\diag_{\al\in A}f_\al\: X\to \prod_{\al\in A} Y_\al,\ x\mapsto (f_\al(x))_{\al\in A}
\]
is called the \term{diagonal mapping} of the family of mappings $f_\al$, $\al\in A$. If $A=\{1,2,\dots,n\}$, 
then we denote
\[
\diag_{n\in A}f_n=\diag_{i=1}^n f_n\: X\to \prod_{i=1}^n Y_n,\ x\mapsto (f_1(x),f_2(x),\dots,f_n(x)).
\]

\subsubsection{Definitions from measure theory}

Let $\mu$ be a probability measure on the Borel $\si$-algebra of a space $X$.
A measure $\mu$ is said to be \term{Radon} if for every Borel set $B\subset X$ and every $\ep>0$, 
there exists a compact set $K\subset B$ such that $\mu(B\setminus K)<\ep$.
The set of Radon probability measures on a space $X$ is denoted by $P(X)$.
The \term{support} $\supp\mu$ of a measure $\mu$ is the smallest closed set $F\subset X$ for which $\mu(F)=1$. 
A measure $\mu$ is said to be \term{discrete} if $\mu(C)=1$ for some
countable $C\subset X$. A measure $\mu$ is \term{continuous} if $\mu(\sset x)=0$ for every $x\in X$.

%Кроме радоновских мер рассматриваются следующие семейства мер.
We will consider families of probability measures on a space $X$ that are identified with certain families
of Radon probability measures on the Stone--\v Cech compactification $\bt X$.
Let $n\in\N$. For $k\in\sset{n,f,\om,\be,R,\tau,\si,a}$ and $\mu\in P(\be X)$, 
we say that $\mu\in P_k(X)$ if the following condition $(P_k)$ is satisfied:
\begin{itemize}
\item[$(P_n)$]
$|\supp \mu|\leq n$;
\item[$(P_f)$]
$|\supp \mu|<\om$;
\item[$(P_\om)$]
the measure $\mu$ is discrete;
\item[$(P_\be)$]
$\mu(K)=1$ for some compact set $K\subset X$;
\item[$(P_R)$]
$\mu(S)=1$ for some $\si$-compact set $S\subset X$;
\item[$(P_\tau)$]
$\mu(K)=0$ for all compact sets $K\subset \bt X\setminus X$;
\item[$(P_\si)$]
$\mu(K)=0$ for all compact sets $K\subset \bt X\setminus X$ of type $G_\de$ in $\bt X$;
\item[$(P_a)$]
$\mu$ is arbitrary.
\end{itemize}

Note that $P(X)=P_R(X)$, $P(\bt X)=P_a(X)$,
$P_f(X)=\bigcup_{n=1}^\infty P_n(X)$, $P_n(X)\subset P_f(X)
\subset P_\be(X)\subset P_\tau(X)\subset P_\si(X)\subset P_a(X)$, and $P_f(X)\subset P_\om(X)\subset P_R(X)$.
The space $P_1(X)$ consists of all Dirac measures on $X$ and is identified with $X$.
The space $P_\si(X)$ is identified with the space of probability measures on the Baire $\si$-algebra of $X$.
The space $P_\tau(X)$ is identified with the space of $\tau$-additive probability measures on $X$.
A discussion of some of these spaces can be found in~\cite{BanakhChigogidzeFedorchuk2003}.

For $k\in\sset{f,\om,\be,R,\tau,\si,a}$, $n\in\N$, and a space $X$, we denote by $P_k^n(X)$ 
the iterated spaces of measures defined by 
$P_k^1(X)=P_k(X)$ and $P_k^n(X)=P_k(P_k^{n-1}(X))$ for $n>1$.

Let $\ph\: X\to Y$ be a continuous mapping of spaces. The mapping $P(\ph)\: P(X)\to P(Y)$ is defined by 
$\mu(\ph^{-1}(B))=P(\ph)(\mu)(B)$ for $\mu\in P(X)$ and a Borel set $B\subset Y$.
The mapping $P(\ph)$ is a continuous extension of $\ph$ \seee\cite[2.6]{Bogachev2016book}.
For compact spaces $X$ and $Y$, the mapping $P(\ph)$ is defined by
\[
\int\limits_Y f(y)\,\mathrm{d} \nu = \int\limits_X f(\ph(x)) \,\mathrm{d} \mu,
\]
where $f$ is a continuous function on $Y$, $\mu\in P(X)$, and $\nu = P(\ph)(\mu)$. The mapping $\ph$ 
extends to a continuous mapping $\bt \ph\: \bt X\to \bt Y$, and 
$P(\bt \ph)$ maps $P(\bt X)$ to $P(\bt Y)$.
For $k\in\sset{f,\om,\be,R,\tau,\si,a}$, we have $P(\bt \ph)(P_k(X))\subset P_k(Y)$ 
\see\cite{Banakh1995ru,Fedorchuk1991}.

\subsubsection{Definitions from topology}

For a space $X$ and $M\subset X$, we denote by $\clx XM$ the closure of $M$ in $X$. 
If it is clear from the context which $X$ is meant, then we write $\cl M$ instead of~$\clx XM$.

A space $X$ is called a \term{$k$-space} if $F\subset X$ is closed if and only if $F\cap K$ is closed for every 
compact set $K\subset X$. All locally compact spaces are $k$-spaces.

A space $X$ is said to have \term{countable Suslin number} if every disjoint family of open sets in $X$ 
is at most countable.

A family $\cU$ of subsets of a space $X$ is said to be \term{locally finite} if, for any point $x\in X$, 
there exists a neighborhood $V$ of $x$ such that $|\set{U\in \cU: U\cap V\neq\es}|<\om$.

A sequence $(M_n)_n$ of subsets of $X$ \term{accumulates} to a point $x\in X$ if 
$|\{n\in\om: U_n\cap U\neq\es\}|=\om$ for every neighborhood $U$ of $x$.

A space $X$ is said to be \term{pseudocompact} if every continuous function on $X$ is bounded. Below we list 
useful pseudocompactness criteria (see \cite[1.1.4, 1.1.5, and 1.1.7]{AACCICTM2018}, \cite[3.10.22 and 
3.10.23]{EngelkingBookRu}).
%\begin{fact}\label{f:defs:pc}
For a space $X$, the following conditions are equivalent:
\begin{enumerate}
\item
$X$ is pseudocompact;
\item
If $(U_n)_n$ is a sequence of nonempty open sets $X$, then $(U_n)_n$ accumulates to some point $x\in X$;
\item
If $(U_n)_n$ is a decreasing sequence of nonempty open sets $X$, then $\bigcap_n \cl{U_n}\neq\es$;
\item
Every locally finite family of open sets is finite.
\end{enumerate}
%\end{fact}

A space $X$ is said to be \term{$\om$-bounded} ($\om$-bounded) if $\cl M$ is compact for every countable 
$M\subset X$. A space $X$ is said to be \term{totally countably compact} (totally countably compact) if for every 
infinite $L\subset X$ there exists a countable $M\subset L$ such that $\cl M$ is compact 
\see\cite{Vaughan1984handbook}.

A space $X$ is called \term{pseudo-$\om$-bounded} if for every
sequence $(U_n)_n$ of nonempty open sets, there exists a compact subset $K\subset X$ such that 
$U_n\cap K\neq\es$ for all $n$ 
\see\cite{AngoaOrtiz-CastilloTamariz-Mascarua2013,AngoaOrtiz-CastilloTamariz-Mascarua2014,Garcia-FerreiraOrtiz-Castillo2018}.

A space $X$ is called \term{almost pseudo-$\om$-bounded} if for every
sequence $(U_n)_n$ of open nonempty open sets, there exists a subsequence $(n_k)_k$ and a compact 
subset $K\subset X$ such that $U_{n_k}\cap K\neq\es$ for all $k$ 
\cite[Definition 3.8]{AngoaOrtiz-CastilloTamariz-Mascarua2014}, \cite[Definition 1.4.5]{AACCICTM2018}.

A space $X$ is called a \term{Frol\'\i k space} if $X\times Y$ is pseudocompact for every pseudocompact space $Y$.

%Every Frolik space is pseudocompact. The product of any family of Frolik spaces is a Frolik space \cite[Theorem 3.1]{Noble1969-2}.

A class of spaces $\cP$ is said to be \term{multiplicative} if $\prod\cM\in\cP$ for every family 
$\cM\subset \cP$.

Results of \cite{Vaughan1984handbook,AACCICTM2018,Garcia-FerreiraOrtiz-Castillo2018}
imply the following assertion.

\begin{proposition}
\label{p:intro:defs:1}
{
\long\def\tttT#1{{\parbox[c][2em]{16em}{{\begin{center}\small #1\end{center}}}}}
\long\def\tttt#1{{\parbox[c][2em]{12em}{{\begin{center}\small #1\end{center}}}}}
%\shorthandoff{"}
\[
\begin{tikzcd}
\tttt{countably compact spaces}
\ar[d,xshift=2ex,"{k\text{-spaces}}"]
\\
\tttT{totally countably compact spaces} \arrow[d] \arrow[u]
&
\tttt{$\om$-bounded spaces} \arrow[d] \arrow[l]
\\
\tttT{almost pseudo-$\om$-bounded spaces} \arrow[dr]
&
\tttt{pseudo-$\om$-bounded spaces} \arrow[l]
\\
\tttt{pseudocompact spaces} \ar[u,"{k\text{-spaces}}"]
&
\tttt{Frol\'\i k spaces} \arrow[l]
\end{tikzcd}
\]
%\shorthandon{"}
}
The classes of $\om$-bounded, pseudo-$\om$-bounded, almost pseudo-$\om$-bounded, and Frol\'\i k spaces 
are multiplicative.
\end{proposition}

A space $X$ is called a Mr\'owka space if (1)~the set of isolated points in $X$ is countable and dense in $X$;
(2)~the set of nonisolated points is discrete; (3)~$|\bt X \setminus X|=1$. Mr\'owka \see\cite{Mrowka1977} proved
that such spaces exist (see also \cite[Theorem 8.6.1]{Hernandez-HernandezHrusak2018}).
Mr\'owka spaces are pseudocompact and locally compact \see\cite{Hernandez-HernandezHrusak2018}.

A subset $M\subset X$ of $X$ is \term{relatively countably compact} in $X$ if every sequence $(x_n)_n\subset M$
has a limit point in~$X$.
A subset $G\subset X$ of $X$ is called a \term{$G_\de$-set} if it is the intersection of countably many open sets.

A point $x$ of $X$ is called a \term{$P$-point} if $x$ belongs to the interior of any $G_\de$-set containing~$x$.

A subset $S\subset X$ of $X$ is said to be \term{$G_\de$-dense} in $X$ if $S$ intersects any nonempty 
$G_\de$-set. A space $X$ is pseudocompact if and only if $X$ is $G_\de$-dense in the Stone--\v Cech 
compactification of~$\bt X$.

A space $X$ is said to be \term{realcompact} if $X$ can be embedded in a product of real lines $\R^A$ as a closed 
subspace. The Hewitt realcompactification $\nu X$ of a space $X$ is a realcompact extension of $X$ 
such that $X$ is dense and $C$-embedded in $\nu X$ \seee\cite[Section 3.11]{EngelkingBookRu}.
The Hewitt realcompactification $\nu X$ is naturally embedded in the Stone--\v Cech compactification $\bt X$: 
$p\in \nu X\subset \bt X$ 
if and only if every $G_\de$-set $G\subset \bt X$ containing $p$ intersects $X$.
A space $X$ is pseudocompact if and only if $\nu X=\bt X$. A space $X$ is realcompact if and only if 
$\nu X=X$ \seee\cite[Section 3.11]{EngelkingBookRu}.

Other names for realcompact spaces are $R$-complete spaces, $Q$-spaces,  real complete spaces, 
functionally complete spaces, $R$-compact spaces, Hewitt--Nachbin spaces, and Hewitt complete spaces; 
Hewitt extension is also called Hewitt completion.

A family of open subsets $\cN$ of a space $X$ is called a \term{$\pi$-base} if for every nonempty open set
$U\subset X$, there exists a nonempty $V\in\cN$ such that $V\subset U$.

A space $X$ is Lindel\"of if every open cover of $X$ has a countable subcover.

A space $X$ is locally (path) connected if the open (path) connected subsets of $X$ form a base for~$X$.

A family $\cM$ of subsets of a space $X$ is said to have the \term{finite intersection property} 
if the intersection $\bigcap\cL$ is nonempty for every nonempty finite subfamily $\cL\subset\cM$.
A space $X$ has \term{precaliber} $\om_1$ if for every uncountable family $\cU$ of open subsets of $X$, 
there exists an uncountable subfamily $\cV\subset\cU$ with the finite intersection property.

A space $X$ is said to be \term{scattered} if any non-empty subset $M\subset X$ contains a point $x\in M$ 
isolated in~$M$.

%%% end file('tex/sec/defs.tex') %%%

\subsection{Main Results}

A space $\bt X$ is locally connected if and only if $X$ is locally connected and pseudocompact 
\cite[Theorem 9.3]{Walker1974}. Based on this fact and on the fact that a convex subset
of a topological vector space (TVS) is locally connected (Proposition~\ref{f:defs:circled}), the following fact 
was essentially proved in \cite{Bogachev2024dan-ru} (see the proof of Theorem~2 in~\cite{Bogachev2024dan-ru}).

\begin{proposition}[{\cite{Bogachev2024dan-ru}}]\label{p:intro:1}
Let $X$ be a space for which $\bt X$ be homeomorphic to a convex subset of some TVS.
Then $X$ is pseudocompact.
\end{proposition}

This proposition also follows from the following fact: if $\bt X$ is path-connected, 
then $X$ is pseudocompact \cite{Reznichenko2024betacs-ru}.
The following proposition is a partial converse of this assertion.

\begin{proposition}[\cite{Reznichenko2024betacs-ru}]\label{p:intro:2}
If $X$ is a pseudocompact convex subset of a TVS, 
then $\bt X$ is path-connected.
\end{proposition}

The space $\bt X$ in Proposition~\ref{p:intro:2} has the topological-algebraic structure of a convex set 
\see\cite{Reznichenko2024betacs-ru}.

The following three theorems are the main results of this paper.\footnote{Theorems~\ref{t:intro:1}, 
\ref{t:intro:2}, and~\ref{t:intro:4} are proved in Section~\ref{sec:main}.}

\begin{theorem}\label{t:intro:1}
Let $X$ be a space such that $X^n$ is pseudocompact for $n\in\N$, and let $Y$ be a pseudocompact 
subspace of $P(\bt X)$ containing $P_f(X)$.
Then $Y$ is $C^*$-embedded in $P(\bt X)$ (i.e.,~$\bt Y=P(\be X)$).
\end{theorem}

\begin{theorem}\label{t:intro:2}
If $P(X)$ is pseudocompact, then $X^n$ is pseudocompact for all $n\in\N$.
\end{theorem}

These two theorems and Bogachev's theorem \cite[Theorem~2]{Bogachev2024-1ru} imply 
the following theorem, which answers Bogachev's question.

\begin{theorem}\label{t:intro:3}
The space $P(X)$ is $C^*$-embedded in $P(\bt X)$ (i.e.,~$\bt P(X)=P(\be X)$) if and only if $P(X)$ 
is pseudocompact. In this case, $X^n$ is pseudocompact for all $n\in\N$.
\end{theorem}

From Theorem \ref{t:intro:3} it follows that the problem of characterizing $X$ 
for which $\bt P(X)=P(\bt X)$, which was stated in \cite{Bogachev2024-1ru}, 
is equivalent to the following question.

\begin{problem}[\cite{Bogachev2024-1ru}]\label{pq:intro:main:1}
Characterize those spaces $X$ for which $P(X)$ is pseudocompact.
\end{problem}

Gliksberg \see\cite{gli1959} proved that for infinite spaces $X$ and $Y$
the equality $\bt X\times Y=\bt (X\times Y)$ holds if and only if $X\times Y$ is pseudocompact.
Gliksberg's theorem gives fundamental importance to the question of pseudocompactness of products 
in the study of Stone--\v Cech compactifications and extension operations, in particular, in the present paper 
(see Theorem~\ref {t:intro:1}). The Comfort--Ross theorem \cite{ComfortRoss1966} states that any product 
of pseudocompact topological groups is pseudocompact.
In \cite{ReznichenkoUspenskij1998}, this result was extended to Mal'tsev spaces.

In \cite{Reznichenko2022,Reznichenko2022-2} it was proved that if $X_1$, $X_2$, \dots, $X_n$ are pseudocompact 
spaces, then any continuous function on $\prod_{i=1}^n X_i$ can be extended to a separately continuous 
function on $\prod_{i=1}^n \bt X_i$
(in \cite{Reznichenko1994} this theorem was proved for $n=2$ and in \cite{ReznichenkoUspenskij1998}, for $n=3$).
In some situations, this theorem allows one to get rid of the requirement
that the finite powers of the space under consideration 
be pseudocompact \see\cite{Reznichenko1994,ReznichenkoUspenskij1998}.
In the present paper, we use this theorem to prove Proposition~\ref{p:intro:2} and the following theorem 
(cf.\ Theorem~\ref{t:intro:1}).

\begin{theorem}\label{t:intro:4}
Let $X$ be a pseudocompact space with precaliber $\om_1$ (e.g., such that $\bt X$ is separable),
and let $Y\subset P(\bt X)$ be a pseudocompact space containing $P_f(X)$. Then $Y$ is $C^*$-embedded in 
$P(\bt X)$ (i.e.,~$\bt Y=P(\be X)$).
\end{theorem}

\begin{problem}\label{pq:intro:th1and4}
Let $X$ be a pseudocompact space, and let $Y\subset P(\bt X)$ be a pseudocompact space containing $P_f(X)$. 
Is it true that $Y$ is $C^*$-embedded in~$P(\bt X)$?
\end{problem}

\subsection{Examples of Spaces $X$ for Which $P(X)$ is Pseudocompact and not Pseudocompact}

%В разделе \ref{sec:other} для $k\in\sset{f,\om,\be,\si}$ дается характеристика пространств $X$, для которых $P_k(X)$ псевдокомпактно.

\begin{example}\label{e:intro:example:1}
Let $X$ be an infinite space such that $X^n$ is countably compact for any $n\in\N$ and 
all compact subsets of $X$ are finite. For example, the space $X$ in \cite[Proposition 16]{rezn2020} has these 
properties. Proposition~\ref{p:intro:example:1} implies that $P(X)$ is not pseudocompact.
\end{example}

Example~\ref{e:intro:example:1}, together with Theorem~\ref{t:intro:3}, answers Questions~1 and~2 
of~\cite{Bogachev2024-1ru}.

\begin{example}\label{e:intro:example:2}
Let $X$ be any of the following spaces:
\begin{enumerate}
\item
the space $\om_1$ of all countable ordinals;
\item
the Tychonoff plane $\Pi_T=((\om_1+1)\times (\om+1))\setminus \sset{(\om_1,\om)}$;
\item
$X_{A,p}=[0,1]^A\setminus \sset p$, where $A$ is uncountable and $p\in [0,1]^A$.
\end{enumerate}
Then $X$ is a pseudo-$\om$-bounded space (Proposition~\ref{p:example:7}).
Therefore (Theorem~\ref{t:intro:other:3}), the space $X$, the space of measures $P(X)$, and 
the iterated space of measures $P^n(X)$
are pseudo-$\om$-bounded and hence pseudocompact to any power. Thus, by Theorem~\ref{t:intro:3},
$P(X)$ is $C^*$-embedded in $P(\bt X)$ (i.e.,~$\bt P(X)=P(\be X)$).
\end{example}

In \cite{Bogachev2024-1ru}, it was proved that $\bt P(X)=P(\bt X)$ if $X\in\sset{\om_1,\Pi_T}$.

\begin{example}\label{e:intro:example:3}
There exists a Mr\'owka space $X$ for which $P(X)$ is not pseudocompact (Theorem~\ref{t:Mrowka:1}).
The space $X$ is pseudocompact to any power and locally compact, and the Stone--\v Cech compactification 
of $\bt X$ is scattered.
\end{example}

This example answers a question of~\cite{Bogachev2024-1ru}.

\begin{problem}\label{pq:intro:example:1}
Does there exist a Mr\'owka space $X$ for which $P(X)$ is pseudocompact?
\end{problem}

In \cite[Example~2]{Bogachev2024-1ru}, it was proved that if $p\in\bt \om\setminus \om$ is a $P$-point
in the remainder $\bt \om\setminus \om$, then the space $P(\bt\om\setminus\sset p)$ is countably compact.
The existence of $P$-points in $\bt \om\setminus \om$ does not depend on the ZFC axioms: 
there are ZFC models in which they exist, and there are models
in which they do not exist \see\cite{vanMill1984handbook}.

\begin{example}\label{e:intro:example:4}
There exists an ultrafilter $p\in\bt \om\setminus \om$ such that the space $P(\bt\om\setminus\sset p)$ is totally 
countably compact (Propositions~\ref{p:example:2} and~\ref{p:example:5}).
\end{example}

\begin{problem}\label{pq:intro:example:2}
Does there exist an ultrafilter $p\in\bt \om\setminus \om$ for which the space $P(\bt\om\setminus\sset p)$ is 
not countably compact?
\end{problem}

At the second step of the proof of \cite[Example~2]{Bogachev2024-1ru}, the following statement was essentially 
proved (see also Proposition~\ref{p:example:3}): if $p\in\bt \om\setminus \om$, then $P_f(\om)$ is dense
and relatively countably compact in $P(\bt\om\setminus\sset p)$. Therefore, $P(\bt\om\setminus\sset p)$ 
is pseudocompact.

\begin{example}\label{e:intro:example:4}
Let $X=\bt\om\setminus\sset p$, where $p\in\bt \om\setminus \om$. Then the space $P(X)$ is almost 
pseudo-$\om$-bounded (Proposition~\ref{p:example:4}) and, therefore, pseudocompact. Moreover, 
$X$ is a locally compact pseudocompact space.
The space $X$ is not pseudo $\om$-bounded, since the closure of the set $\om$ of isolated points is not compact.
\end{example}

\begin{problem}\label{pq:intro:example:3}
Let $X=\bt\om\setminus\sset p$, where $p\in\bt \om\setminus \om$. Are the iterated spaces of measures $P^n(X)$ 
pseudocompact?
\end{problem}

\subsection{Extensions of Other Spaces of Probability Measures}

The space $P_\si(X)$ of Baire measures is real complete \see\cite{BanakhChigogidzeFedorchuk2003}.
The spaces $P_\tau(X)$ and $P(X)$ are dense in $P_\si(X)$. In \cite[Section 5]{BanakhChigogidzeFedorchuk2003}, 
the following problem was posed: find conditions under which $P_\si(X)$ is the Hewitt extension of $P_\tau(X)$ 
and $P(X)$. This problem is equivalent to that of finding conditions under which $P_\tau(X)$ and $P(X)$ 
are $C$-embedded in $P_\si(X)$. 
The following problem can be viewed as an extension of problems of Bogachev (Problem~\ref{pq:intro:main:1}) 
and of Banakh, Chigogidze, and Fedorchuk \cite[Problem~5.5]{BanakhChigogidzeFedorchuk2003}.

\begin{problem}\label{pq:intro:other:2}
Given a space $X$, $n\in\N$, and $k,l\in \sset{n,f,\om,\be,R,\tau,\si,a}$ such that $P_k(X)\subset P_l(X)$, 
find conditions on $X$ under which (1)~$P_k(X)$ is $C^*$-embedded in $P_l(X)$; 
(2)~$P_k(X)$ is $C$-embedded in $P_l(X)$; (3)~$P_k(X)= P_l(X)$.
\end{problem}

A space $X$ is said to be \term{universally measurable} if $P_\si(X)=P(X)$
\cite[Section~2]{BanakhChigogidzeFedorchuk2003}.
Item (3) of question \ref{pq:intro:other:2} was considered in \cite[Section~2]{BanakhChigogidzeFedorchuk2003}, where 
spaces $X$ for which $P_\si(X)=P_\tau(X)$ were studied.

\begin{theorem}[(Section \ref{sec:other})]\label{t:intro:other:1}
For a space $X$, the following conditions are equivalent:
\begin{enumerate}
\item
$X$ is pseudocompact;
\item
$P_\si(X)$ is pseudocompact;
\item
$P_\si(X)$ is $C^*$-embedded in $P(\bt X)$;
\item
$P_\si(X)$ is compact and coincides with $P(\bt X)$.
\end{enumerate}
\end{theorem}

In \cite[Problem~5.5]{BanakhChigogidzeFedorchuk2003}, the following question was asked: Is $P_\si(X)$ the Hewitt
extension of $P_\tau(X)$? This question can be reformulated as follows: Is $P_\tau(X)$ $C$-embedded in $P_\si(X)$?
The following example shows that the answer to this question is negative.

\begin{example}\label{e:intro:other:1}
Let $X$ be a Mr\'owka space for which $P(X)$ is not pseudocompact (see Example~\ref{e:intro:example:3}).
The space $X$ is pseudocompact and locally compact. Since $X$ is locally compact, it follows that 
$P(X)=P_\tau(X)$, and since $X$ is pseudocompact, it follows that $P_\si(X)=P(\bt X)$ 
(Theorem~\ref{t:intro:other:1}). But $P(X)$ is not $C$-embedded in the compact 
space $P(\bt X)$, because $P(X)$ is not pseudocompact.

Thus, $P(X)=P_\tau(X)$, $P_\si(X)=P(\bt X)$, and $P_\tau(X)$ is not $C$-embedded in $P_\si(X)$, i.e., 
$P_\si(X)$ is not the Hewitt extension of $P_\tau(X)$.
\end{example}

\begin{proposition}\label{p:other:intro:1}
If $P_\tau(X)$ is pseudocompact, then so is $X$.
\end{proposition}

\begin{proof}
%[Proof of Proposition \ref{p:other:intro:1}]
Since $P_\tau(X)$ is pseudocompact and dense in $P_\si(X)$, it follows that $P_\si(X)$ is pseudocompact. 
Theorem~\ref{t:intro:other:1} implies that $X$ is pseudocompact.
\end{proof}

\begin{problem}\label{pq:intro:other:3}
Let $X$ be a space for which $P_\tau(X)$ is pseudocompact.
\begin{enumerate}
\item
Is $P(X)$ pseudocompact?
\item
Is $X^n$ pseudocompact for all $n\in\N$?
\item
Is $P_\tau(X)$ \,$C^*$-embedded in $P(\bt X)$?
\end{enumerate}
\end{problem}

If the answer to Question~\ref{pq:intro:other:3}(1) is positive, then so is the answer to 
Question~\ref{pq:intro:other:3}(2)  (Theorem~\ref{t:intro:2}). If the answer to 
Question~\ref{pq:intro:other:3}(2) is positive, then so is the answer to Question~\ref{pq:intro:other:3}(3) 
(Theorem~\ref{t:intro:1}).

Theorem~\ref{t:intro:4} and Proposition~\ref{p:other:intro:1} imply the following assertion.

\begin{theorem}\label{t:other:intro:2}
If $X$ has precaliber $\om_1$ and $P_\tau(X)$ is pseudocompact, then $P_\tau(X)$ is $C^*$-embedded in~$P(\bt X)$.
\end{theorem}

\begin{theorem}[(Section \ref{sec:other})]\label{t:intro:other:3}
For a space $X$, the following conditions are equivalent:
\begin{enumerate}
\item
$X$ is pseudo-$\om$-bounded;
\item
$P_\be(X)$ is pseudocompact;
\item
$P_\be(X)$ is $C^*$-embedded in $P(\bt X)$;
\item
$P_\be(X)$ is pseudo-$\om$-bounded.
\end{enumerate}
\end{theorem}

\begin{problem}\label{pq:intro:other:4}
Let $X$ be a space for which $P(X)$ ($P_\tau(X)$) is pseudocompact.
\begin{enumerate}
\item
Is $X$ a Frol\'\i k space?
\item
Is $P(X)$ ($P_\tau(X)$) a Frol\'\i k space?
\end{enumerate}
\end{problem}

\subsection{Spaces of Signed Measures}

Given a compact space $K$, we denote by $M(K)$ the space of all Radon measures on $K$ with the weak topology.
Then $M(K)$ is a locally convex space and contains $P(K)$.
Let $X$ be a space.
For $k\in\sset{f,\om,\be,R,\tau,\si,a}$ we introduce the following notation:
\begin{align*}
M_k(X) &= \set{t\mu+ q\nu: \mu,\nu\in P_k(X),\ 0\leq t,q},
\\
U_k(X) &= \set{t\mu+ q\nu: \mu,\nu\in P_k(X),\ 0\leq t,q\leq 1}.
\end{align*}
The space $U_k(X)$ is the unit ball of $M_k(X)$.
The spaces $U_k(X)$ and $M_k(X)$ for $k\in\sset{\be,R,\tau}$ were studied in the papers
\cite{Koumoullis1981,Sadovnichii1999,SadovnichiiFedorchuk1999,Sadovnichii2000,Sadovnichii2007,SadovnichiiFedorchuk2008,Sadovnichii2020}.

The following question is a development of Question~\ref{pq:intro:other:2}.

\begin{problem}\label{pq:intro:am:1}
Given a space $X$,  $n\in\N$, and $k,l\in \sset{n,f,\om,\be,R,\tau,\si,a}$ such that $P_k(X)\subset P_l(X)$, 
find conditions on $X$ under which (1)~$M_k(X)$ is $C^*$-embedded in $M_l(X)$; (2)~$M_k(X)$ is  
$C$-embedded in $M_l(X)$;
(3)~$U_k(X)$ is $C^*$-embedded in $U_l(X)$; (4)~$U_k(X)$ is $C$-embedded in~$U_l(X)$.
\end{problem}

\subsection{Dense Pseudocompact Subsets of Convex Compact Sets}

For $i=1,2,3$, we denote by $\cc i$ the class of convex compact sets $K$ for which the following condition \ccc i 
holds:
\begin{enumerate}
\item[\ccc 1]
$K=P(X)$, where $X$ is a compact set;
\item[\ccc 2]
$K$ is a convex compact subset of some locally convex space (LCS);
\item[\ccc 3]
$K$ is a convex compact subset of some TVS.
\end{enumerate}

If $X$ is a pseudocompact convex subset of a convex compact set, then $\bt X$ is path-connected 
(Proposition~\ref{p:intro:2}).

\begin{problem}\label{pq:intro:conv:1}
Let $i=1,2,3$, and let $Y$ be a dense pseudocompact subset of a convex compact set $K\in\cc i$.
\begin{enumerate}
\item
Is $Y$ connected?
\item
Is $\bt Y$ path-connected?
\end{enumerate}
\end{problem}

Theorems \ref{t:intro:1} and \ref{t:intro:4} show that dense pseudocompact subsets of convex compact sets
are sometimes $C^*$-embedded in these convex compact sets.

\begin{example}\label{e:intro:conv:1}
Let $X=\sset{0,1}\times \om_1$, and let $\al X$ be the one-point Alexandroff compactification of $X$. We set 
$K=P(\al X)$. Since $\om_1$ is \cccc/-bounded (Proposition~\ref{p:example:om1}), it follows that so is $X$.
Consequently, the space of measures $P(X)$ is $\om$-bounded (Proposition~\ref{p:example:1}) and, therefore, 
countably compact. The space $P(X)$ is a dense pseudocompact convex subset of the convex compact set~$K$.

The set $X$ is not $C^*$-embedded in $\al X$, because the characteristic function $f$ of the set 
$\sset{0}\times \om_1$ is continuous on $X$ and does not extend to a continuous function on $\al X$.
Proposition~\ref{a:tms:0} implies that $P(X)$ is not $C^*$-embedded in $K=P(\al X)$.
\end{example}

\begin{problem}\label{pq:intro:conv:2}
Let $i=1,2,3$, and let $X$ be a dense (convex) pseudocompact subset of a convex compact set $K\in\cc i$.
\begin{enumerate}
\item
Is it true that $\bt X$ is homeomorphic to some space in $\cc i$?
\item
Is $X$ a Frol\'\i k space?
\item
Is it true that the finite powers of $X$ are pseudocompact?
\item
Let $Y$ be a dense (convex) pseudocompact subset of a convex compact set in $\cc i$.
Is it true that $X\times Y$ is pseudocompact?
\end{enumerate}
\end{problem}

%%% end file('tex/sec/intro.tex') %%%

\section{Topology of Spaces of Measures}\label{sec:tms}
%%% begin file('tex/sec/tms.tex') %%%
% tms
Let $X$ be a space. For $x\in X$, we denote by $\de_x$ the Dirac probability measure on $X$:
$\de_x(\sset x)=1$ and $\de_x(X\setminus \sset x)=0$. We set
\[
\de^X: X\to P(X),\ x \mapsto \de_x.
\]
The mapping $\de^X$ embeds the space $X$ in $P(X)$ as a closed subspace.

Let $C^*(X)$ denote the set of continuous bounded functions on $X$. For $f\in C^*(X)$, we set
\[
\hat f: P(X)\to \R,\ \mu \mapsto \int_X f(x) \,\mathrm{d} \mu.
\]
Functions of the form $\hat f$ are continuous affine functions on $P(X)$. 
The family  $\set{\hat f: f\in C^*(X)}$ of functions
generates the topology of $P(X)$ (i.e., the family $\{\hat f^{-1}(U): f\in C^*(X),\, U\subset \R$ is open$\}$ 
is a subbase of~$P(X)$).

Identifying $X$ with $\de^X(X)=P_1(X)$, we can say that $X$ is $C^*$-embedded in $P(X)$.

\begin{assertion}\label{a:tms:0}
If $Y\subset X$ and $P(Y)$ is $C^*$-embedded in $P(X)$, then $Y$ is $C^*$-embedded in $X$.
\end{assertion}

\begin{proof}
Let $f\in C^*(Y)$. Then $\hat f\in C^*(P(Y))$. Since $P(Y)$ is $C^*$-embedded in $P(X)$, it follows that 
the function $\vh f$ extends to a continuous
bounded function $g\: P(X)\to\R$. Therefore, the function $g\circ \de^X\in C^*(X)$ extends the function~$f$.
\end{proof}

A subset $G\subset X$ is said to be \term{functionally open} (\term{functionally closed}) if there exists a
continuous function $f\: X\to \R$ such that $G=f^{-1}(\R\setminus \sset 0)$ (respectively, $G=f^{-1}(\sset 0)$).

The sets of the form
\[
W(U,\ep)=\set{\mu\in P(X): \mu(U)>\ep },
\]
where $U$ is a functionally open subset of $X$ and $\ep>0$, 
form a subbase of the topology of the space of measures $P(X)$ \seee\cite[4.3]{Bogachev2016book}.

Note that the set $W(U,\ep)$ is open whenever $U$ is open and $\ep>0$. Indeed, let $\mu\in W(U,\ep)$.
Since $\mu$ is a Radon measure, we have $\mu(K)>\ep$ for some compact set $K\subset U$. Let $f$ be a
continuous function $f\: X\to[0,1]$ such that $f(X\setminus U)=\sset 0$ and $f(K)=\sset 1$.
Then $\mu\in W(U_0,\ep)\subset W(U,\ep)$ for $U_0=f^{-1}(\R\setminus \sset 0)$.

Let $U_1,U_2,\dots,U_n\subset X$, and let $\ep_1,\ep_2,\dots,\ep_n$ be positive numbers. We set 
\[
W( (U_i,\ep_i)_{i=1}^n ) = \bigcap_{i=1}^n W(U_i,\ep_i).
\]
The sets $W( (U_i,\ep_i)_{i=1}^n )$ for functionally open $U_i$ and $\ep_i>0$ form a base 
for the space $P(X)$. From this description of the topology of $P(X)$ it immediately follows that 
if $Y\subset X$, then the embedding of $P(Y)$ into $P(X)$ is topological \seee\cite[Theorem~2.4]{Banakh1995ru}.

For a family of subsets $\cM$ of $X$, let $\prt \cM$ denote a minimal partition $\cP$ of $X$ such that if 
$M\in\cM$, $P\in\cP$, and $M\cap P\neq\es$, then $P\subset M$. For $x\in X$, we set
\[
P_x(\cM) = \bigcap\set{M\in \cM: x\in M}\setminus \bigcup\set{M\in \cM: x\notin M}.
\]
Then $\prt \cM=\set{P_x(\cM):x\in X}$.

If $\cM$ is a finite family of Borel sets, then so is~$\prt \cM$.

\begin{assertion}\label{a:tms:1}
Let $U_1,U_2,\dots, U_n$ be open subsets of $X$, and let $\ep_1$, $\ep_2$, \dots, $\ep_n$ be positive numbers. 
If $\mu\in W( (U_i,\ep_i)_{i=1}^n )$,  $\cP$ is a finite partition of a space $X$ consisting of Borel sets
and refining the family $\prt{\set{U_i:i=1,\dots,n}}$, $x_P\in P$ for $P\in\cP$, and
\[
\nu = \sum_{P\in\cP} \mu(P) \de_{x_P},
\]
then $\nu\in W( (U_i,\ep_i)_{i=1}^n )$.
\end{assertion}

\begin{proof}
Since $\mu(P)=\nu(P)$ for $P\in\cP$,  it follows that $\mu(U_i)=\nu(U_i)$ for $i=1,2,\dots,n$.
Therefore, $\nu\in W( (U_i,\ep_i)_{i=1}^n )$.
\end{proof}

Note that Proposition~\ref{a:tms:1} implies that $P_f(X)$ is dense in~$P(X)$.

\begin{assertion}\label{a:tms:2}
Let $\mu=\sum_{i=1}^n \la_i\de_{x_i}\in P_f(X)$, where $x_i\in X$, $\la_i>0$ for $i=1,2,\dots,n$, $x_i\neq x_j$
for distinct $i,j\in\sset{1,2,\dots,n}$, and $\sum_{i=1}^n \la_i=1$.
Then the family of all sets of the form $W( (U_i,\la_i-\ep)_{i=1}^n )$, where $(U_i)_i$ is a disjoint family 
of nonempty functionally open neighborhoods of $x_i$ for $i=1,2,\dots,n$ and 
$0<\ep<\min(\la_1,\la_2,\dots,\la_n)$, forms a base at the point $\mu$ of the space $P(X)$.
\end{assertion}

\begin{proof}
Let $W$ be an open nonempty set in $P(X)$, and let $\mu\in W$. There exists an $m\in\N$, 
nonempty functionally open subsets $V_1$, $V_2$, \dots, $V_m$  of $X$,  and numbers 
$\ga_1,\ga_2, \dots, \ga_m>0$ such that
$\mu\in W( (V_j,\ga_j)_{i=1}^m )\subset W$. Let $(U_i)_{i=1}^n$ be a disjoint
system of functionally open neighborhoods of the points $(x_i)_{i=1}^n$ such that $U_i\subset V_j$ whenever  
$x_i\in V_j$.
Let $M_j=\set{i\in \sset{1,2,\dots,n}: x_i\in V_j}$. Then $\sum_{i\in M_j}\la_i>\ga_j$. We set 
\[
\ep = \frac 1{2n} \min \left\{\vphantom{\sum}\right. \sum_{i\in M_j}\la_i-\ga_j: j=1,2,\dots,m\left.\vphantom{\sum}\right\}.
\]
Then
\[
\mu\in W( (U_i,\la_i-\ep)_{i=1}^n ) \subset W( (V_j,\ga_j)_{i=1}^m )\subset W.
\]
\end{proof}

Proposition \ref{a:tms:2}, together with the density of $P_f(X)$ in $P(X)$, implies the following statement.

\begin{assertion}\label{a:tms:3}
The family of sets of the form $W( (U_i,\ga_i)_{i=1}^n )$, where $n\in\N$, $(U_i)_{i=1}^n$ is a disjoint family
of nonempty functionally open sets, and $\ga_i>0$ for $i=1,2,\dots ,n$, forms a $\pi$-base of the space~$P(X)$.
\end{assertion}

%%% end file('tex/sec/tms.tex') %%%

\section{Extensions of Other Spaces of Probability Measures}\label{sec:other}
%%% begin file('tex/sec/other.tex') %%%

\begin{proof}[of Theorem~\ref{t:intro:other:1}]
The implications $(4)\rarr(3)$ and $(4)\rarr(2)$ are obvious.

$(1)\rarr(4)$.
Condition (1) implies the equality $\nu X=\bt X$.
Since $P_\si(X)=P_\si(\nu X)$ \cite[Theorem~1.1]{BanakhChigogidzeFedorchuk2003}, it follows that 
$P_\si(X)=P(\bt X)$.

$(4)\rarr(1)$.
Since the closure of $X$ in $P_\si(X)$ is the Hewitt extension of $X$ 
\cite[Theorem~1.2]{BanakhChigogidzeFedorchuk2003},
we have $\nu X=\bt X$, i.e., $X$ is pseudocompact.

$(2)\rarr(4)$. Since $P_\si(X)$ is real complete \cite[Theorem~1.1]{BanakhChigogidzeFedorchuk2003},
it follows that $P_\si(X)$ is compact and coincides with~$P(\bt X)$.

The implication $(3)\rarr(2)$ follows from Proposition~\ref{p:intro:1}. 
\end{proof}

\begin{proposition}\label{p:other:1}
Let $X$ be a space, and let $k\in \sset{f,\om,\be,R,\tau,\si}$. If $P_k(X)$ is pseudocompact, 
then so is $X$.
\end{proposition}

\begin{proof}
The space $P_k(X)$ is dense in $P_\si(X)$. Therefore, $P_\si(X)$ is pseudocompact. Theorem~\ref{t:intro:other:1} 
implies that $X$ is pseudocompact.
\end{proof}

\begin{proof}[of Theorem~\ref{t:intro:other:3}]
The implication $(4)\rarr(2)$ follows from Proposition~\ref{p:intro:defs:1}.
%Implication $(3)\rarr(1)$ follows from Proposition \ref{p:other:1}.
The implication $(3)\rarr(2)$ follows from Proposition~\ref{p:intro:1}.

$(2)\rarr(1)$. Let $(U_i)_i$ be an infinite sequence of nonempty open sets, and let 
$V_n\subset \cl{V_n}\subset U_n$ be a nonempty open set for each $n\in\N$. We set
$W_n=W((V_i,\tfrac 1{2^i})_{i=1}^n)$ for $n\in \N$. Since the space $P_\be(X)$ is pseudocompact,
it follows that the sequence $(W_n)_n$ accumulates to a measure $\mu\in P_\be(X)$, and $K=\supp \mu$ is compact.

We will show that $K\cap \cl{V_n}\neq\es$ for all $n\in\N$. Assume that, on the contrary, $K\cap \cl{V_n}=\es$.
Let $S=W(X\setminus \cl{V_n},1-\tfrac 1{2^n})$.
Then $S\cap W_m\neq \es$ for some $m>n$. Let $\nu\in S\cap W_m$.
Since $\nu\in S$, we have $\nu(V_n)<\tfrac 1{2^n}$, and since $\nu\in W_m$, we have $\nu\in W(V_n,\tfrac 1{2^n})$ 
and, consequently, $\nu(V_n) > \tfrac 1{2^n}$. We have obtained a contradiction.

We have proved that $K\cap \cl{V_n}\neq\es$ for all $n\in\N$. Therefore, $K\cap U_n\neq\es$ for all $n\in\N$.

$(1)\rarr(4)$. Let $(W_n)_n$ be an infinite sequence of open nonempty subsets of $P(X)$. 
According to Proposition~\ref{a:tms:3}, for every $n\in \N$,  there exists a $k_n\in \N$,
a disjoint family of nonempty functionally open sets $(U_{n,i})_{i=1}^{k_n}$, and numbers 
$\ga_{n_i}>0$ for $i=1,2,\dots,k_n$ such that
\[
W((U_{n,i},\ga_{n,i})_{i=1}^{k_n})\subset W_n.
\]
Since $X$ is pseudo-$\om$-bounded, there exists a compact set $K\subset X$ such that $K\cap U_{n,i}\neq\es$
for any $n\in \N$ and $i\in \sset{1,2,\dots,k_n}$. Therefore, $P(K)$ is compact and $P(K)\cap W_n\neq\es$ 
for all $n\in\N$.

$(2)\rarr(3)$. It follows from $(2)\rarr(1)$ that $X$ is pseudo-$\om$-bounded. 
By virtue of Proposition~\ref{p:intro:defs:1},  $X^n$ is pseudocompact for all $n\in\N$. 
Theorem \ref{t:intro:1} implies that $P_\be(X)$ is $C^*$-embedded in $P(\bt X)$.
\end{proof}

\begin{proposition}\label{p:other:2}
Let $X$ be a space. If $P_\om(X)$ is pseudocompact, then $\bt X$ is a scattered compact space and 
$X$ is a scattered pseudocompact space.
\end{proposition}

\begin{proof}
Suppose that $\bt X$ is not scattered. Then $\bt X$ maps continuously onto the interval $[0,1]$ 
\cite[3.8]{Arhangelskii11989}.
Let $f\: \bt X\to [0,1]$ be a continuous surjective mapping.
Since $f$ is surjective, it follows that
\[
P(f)(P_\om(X))\subset P(f)(P_\om(\bt X))=P_\om([0,1])\neq P([0,1]).
\]
Given that $P_\om(X)$ is dense in $P_\om(\bt X)$ and $P(f)$ is continuous, we obtain
\[
P([0,1])=\cl{P_\om([0,1])}\subset \cl{P(f)(P_\om(X))} \subset P([0,1]).
\]
Therefore, $P(f)(P_\om(X))$
is dense in $P([0,1])$ and does not coincide with $P([0,1])$. This is a contradiction, 
since a dense pseudocompact subset $P(f)(P_\om(X))$ must coincide with the metrizable compact set $P([0,1])$.
Consequently, $\bt X$ and $X$ are scattered. Proposition~\ref{p:other:1} implies that $X$ is pseudocompact.
\end{proof}

\begin{proposition}
\label{f:defs:scattered}
Let $X$ be a scattered space. Then the set $X'$ of isolated points of $X$ is dense in $X$. If, in addition, $X$ 
is pseudocompact, then any infinite sequence $(x_n)_n \subset X'$ has an infinite subsequence 
$(x_{n_k})_k$ converging to some point $x\in X\setminus X'$.
\end{proposition}

\begin{proof}
Let $U\subset X$ be a nonempty open set. Since $X$ is scattered, there exists an isolated point $x\in U$. 
We have $\{x\}=U\cap V$ for some open $V\subset X$. Therefore, the point $x$ is isolated in $X$, that is, $x\in 
U\cap X'\neq\es$.

Suppose that $X$ is pseudocompact and $(x_n)_n \subset X'$ is an infinite sequence. Passing to a subsequence, 
we can assume that $x_n\neq x_m$ for distinct $n,m\in \om$.

Since $\{x\}$ is open for $x\in X'$ and $X$ is pseudocompact, it follows that any infinite set $Q\subset X'$ 
accumulates to some point in $X\setminus X'$.

Let $S$ be the set of limit points of the sequence $(x_n)_n$, and let $y\in S$ be an isolated point in $S$. 
Choose a  neighborhood $U$ of $y$ so that $\cl U\cap S=\{y\}$. Let $(x_{n_k})_k$ be an infinite subsequence 
such that $U\cap \{x_n:n\in\om\}=\{x_{n_k}:k\in\om\}$.

We claim that $(x_{n_k})_k$ converges to $y$. Assume the contrary. Let $y_k=x_{n_k}$ for $k\in\om$.
Then the set $Q=\{y_k:k\in\om\}\setminus V$ is infinite for some neighborhood $V$ of $y$. 
Let $z$ be a limit point of $Q$. Then $z\neq y$ and $z\in \cl U\cap S$ in contradiction with $\cl U\cap S=\{y\}$. 
\end{proof}

\begin{proposition}\label{p:intro:example:1}
Let $X$ be an infinite space in which all compact sets are finite. Then $P(X)$ is not pseudocompact.
\end{proposition}

\begin{proof}
Suppose that, on the contrary, $P(X)$ is pseudocompact.
Since all compact sets in $X$ are finite, we have $P_\om(X)=P(X)$. 
By Proposition \ref{p:other:2}, Proposition~\ref{f:defs:scattered},  
$\bt X$ is a scattered compact set and $X$ is a scattered pseudocompact space, 
because $P_\om(X)$ is pseudocompact. 
 Any sequence of isolated points in an infinite pseudocompact scattered space 
has an infinite convergent subsequence by Proposition~\ref{f:defs:scattered}. 
Therefore, $X$ contains an infinite compact set. This is a contradiction.
\end{proof}

%%% end file('tex/sec/other.tex') %%%

\section{Examples}\label{sec:example}
%%% begin file('tex/sec/example.tex') %%%

Let $\oms=\bt\om\setminus \om$ denote the remainder of $\om$ in the Stone--\v Cech compactification $\bt\om$.
The Stone--\v Cech compactification $\bt\om$ is identified in the standard way with the set of 
ultrafilters on $\om$. The set $\om$ is identified with the set of principal ultrafilters on $\om$, 
and the remainder $\oms$ is identified with the set of nonprincipal ultrafilters on~$\om$.

A nonempty family $p$ of subsets of $\om$ is called a \term{filter} if $p$ has the finite intersection property 
and $A\in p$ and $A\subset B\subset\om$ imply $B\in p$. A filter $p$ is said to be \term{free} 
if the intersection $\bigcap p$ of its elements is empty. A filter $p$ is an \term{ultrafilter} if and only if 
the following condition holds: if $A,B\subset\om$ and $A\cup B=\om$, then $A\in p$ or $B\in p$.
An ultrafilter $p$ is said to be \term{principal} if $p=\{A\subset\om: n\in A\}$ for some $n\in\om$; 
it is identified with $n$.

The base of the topology on $\bt\om$ is formed by sets of the form
\[
U(S)=\set{p\in \bt \om: S\in p},
\]
where $S\subset\om$. A local base of the topology at a point $p\in \bt \om$ is formed by sets of the form 
$U(S)$, where $S\in p$.

We say that a space $X$ is \term{\cccc/-bounded} if $\cl M$ is compact for every $M\subset X$ with 
countable Suslin number.

\begin{proposition}\label{p:example:1}
If $X$ is \cccc/-bounded, then $P(X)$ is $\om$-bounded and $P(X)=P_\be(X)$.
\end{proposition}

\begin{proof}
If $\mu\in P(X)$, then $\supp \mu$ has countable Suslin number.
Since $X$ is \cccc/-bounded, it follows that $\supp \mu$ is compact.
Therefore, $P(X)=P_\be(X)$. Let $(\mu_n)_n\subset P(X)$. Then $M=\bigcup_{n=1}^\infty \supp\mu_n$ 
has countable Suslin number. The closure $K=\cl M$ is compact, because $X$ is \cccc/-bounded. 
Hence $P(K)$ is compact and contains~$(\mu_n)_n$.
\end{proof}

\begin{proposition}\label{p:example:om1}
The space $\om_1$ of all countable ordinals is \cccc/-bounded, $P(\om_1)$ is $\om$-bounded, and 
$P(\om_1)=P_\be(\om_1)$.
\end{proposition}

\begin{proof}
In view of Proposition~\ref{p:example:1}, it suffices to show that $\om_1$ is \cccc/-bounded.
Let $Y\subset\om_1$ have countable Suslin number.

We show that $Y$ is separable. By Proposition~\ref{f:defs:scattered} the set of isolated points 
in $Y$ is dense in $Y$, because $Y$ is scattered. In spaces with countable Suslin number, the set of isolated 
points is at most countable.

Since $Y$ is separable, we have $Y\subset \la$ for some $\la<\om_1$. Therefore, $\cl Y$ is compact.
\end{proof}

\begin{proposition}[\cite{Kunen1980,vanMill1982Sixteen}]\label{p:example:2}
There exists a $p\in\oms$ such that $\oms\setminus\sset p$ is \cccc/-bounded.
\end{proposition}

\begin{proof}
The point $x_\es\in\oms$ defined in Section~3.1 of \cite{vanMill1982Sixteen} is as desired.
Its existence also follows from Propositions~1.4 and~2.6 of~\cite{Kunen1980}.
\end{proof}

\begin{lemma}\label{l:example:1}
Let $(\mu_n)_n\subset P(\om)$, and let $\ep>0$.
\begin{enumerate}
\item
There exists a finite set $A\subset \om$, a disjoint sequence $(A_j)_j$ of finite sets, 
and a subsequence $(n_j)_j$ such that $A\cap A_j=\es$ and $\mu_{n_j}(A\cup A_j)> 1-\ep$ for all $j\in\N$.
\item
Every $p\in\oms$ contains an element $S$ such that the set $\set{n\in \N: \mu_n(S)\leq\ep}$ is infinite.
\end{enumerate}
\end{lemma}

\begin{proof}
We prove (1).
Let $\al\om=\om\cup\sset{\infty}$ be the one-point compactification of the space $\om$. 
Since $P(\al\om)$ is a metric compact space, there exists a subsequence of $(\mu_n)_n$ that converges 
to some measure $\nu\in P(\al\om)$. We have $\nu=c\mu+(1-c)\de_\infty$ for some $c\in[0,1]$ and $\mu\in P(\om)$. 
Passing to a subsequence, we assume that $(\mu_n)_n$ converges to~$\nu$.

We construct the subsequence $(n_k)_k$ and the sets $(A_k)_k$ by induction.

If $c=0$, then we set $A=\es$. If $c>0$, then we take a finite $A\subset\om$ such that $\mu(A)> 1-\ep$.
We set $n_1=1$ and choose a finite $A_1\subset \om\setminus A$ such that $\mu_{n_1}(A\cup A_1)> 1-\ep$.

Suppose that $n_i$ and $A_i$ are constructed for $i<j$. We set $B_j=\bigcup_{i=1}^{j-1}A_j$. Since 
\[
\nu(B_j)=c\mu(B_j)\leq \mu(B_j) 1-\mu(\om\setminus B_j)\leq 1-\mu(A)<\ep,
\]
it follows that $W_j=W(\al\om\setminus B_j,1-\ep )$ is an open neighborhood of the measure $\nu$. The convergence 
of $(\mu_n)_n$ to $\nu$ implies the existence of an $n_j>n_{j-1}$ such that $\mu_{n_j}\in W_j$.
We have $\mu_{n_j}(B_j)< \ep$.
Choose a finite set $A_j'\subset \om\setminus B_j$ so that $\mu_{n_j}(A_j')> 1- \ep$.
Let $A_j=A_j'\setminus A$. Then $A_j\cap A=\es$, $A_i\cap A_i=\es$ for $i<j$, and $\mu_{n_j}(A\cup A_j)> 1- \ep$.

We prove (2). Using (1) and passing to a subsequence, we can assume that there exists a finite set $A\subset \om$
and a disjoint family $(A_n)_n$ of finite sets such that $\mu_{n}(A\cup A_n)> 1-\ep$ for all $n\in\N$.
We set $P=\bigcup_{i=1}^\infty A_{2n}$ and $Q=\bigcup_{i=1}^\infty A_{2n-1}$. Since $P\cap Q=\es$, 
it follows that either $P\notin p$ or $Q\notin p$.

Consider the case $P\notin p$; the case $Q\notin p$ is treated similarly.
Let $S=\om\setminus (A\cup P)$. Then $S\in p$ and $S\cap(A\cup A_{2n})=\es$ for $n\in\N$.
Since $\mu_{2n}(A\cup A_{2n})> 1-\ep$, we have $\mu_{2n}(S)<\ep$ for all $n\in\N$.
\end{proof}

The following statement is a generalization of Lemma~\ref{l:example:1}\,(1).

\begin{proposition}\label{p:l:example:1(1)}
Suppose given a space $X$, a sequence $(\mu_n)_n\subset P(X)$, and a number $\ep>0$. 
Then there exists a compact set  $K\subset X$, a disjoint sequence $(K_j)_j$ of compact sets, 
and a subsequence $(n_j)_j$ such that $K\cap K_j=\es$ and $\mu_{n_j}(K\cup K_j)> 1-\ep$ for all $j\in\N$.
\end{proposition}

\begin{proof}
Since $\mu_n$ is a Radon probability measure, there exists a $\si$-compact set $Y_n\subset X$ such that 
$\mu_n(Y_n)=1$. Let $Y=\bigcup_{n=1}^\infty Y_n$. Then the set $Y$ is $\si$-compact and $\mu_n(Y)=1$ for all 
$n\in\N$. If $Y$ is compact, then we set  $K=Y$ and $K_n=\es$. Consider the case where $Y$ is not compact. In this 
case, there exists an increasing (by inclusion) sequence $(C_n)_{n\in\om}$ of compact sets such that 
$\bigcup_{n=0}^\infty C_n=Y$, $S_0=C_0\neq\es$, and $S_n=C_{n}\setminus C_{n-1}\neq\es$ for $n>0$.

We define $\nu_n\in P(\om)$. Let $\nu_n(\{i\})=\mu_n(S_i)$ for $i\in\om$. Lemma~\ref{l:example:1}(1) 
implies the existence of a finite set $A\subset \om$, a disjoint sequence $(A_j)_j$ of finite sets,
and a subsequence $(n_j)_j$ such that $A\cap A_j=\es$ and $\nu_{n_j}(A\cup A_j)> 1-\ep$ for all $j\in\N$.

We can additionally assume that $A=\{0,1,\dots,m\}$ for some $m$. Indeed, $A\subset A'=\{0,1,\dots,m\}$ 
for some $m$. There exists an $N$ such that $A'\cap A_n=\es$ for $n\geq N$. We have $A'\cap A_j'=\es$ and 
$\nu_{n_j'}(A'\cup A_j')> 1-\ep$ for all $j\in\N$, where $n_j'=n_{N+j}$ and $A_j'=A_{N+j}$.

We set $K=C_m$  and  $K_j'=\bigcup\{S_i:i\in A_j\}$ for $j\in\N$. Then 
$\mu_{n_j}(K\cup K_j')=\nu_{n_j}(A\cup A_j)>1-\ep$,  $K_j'$ is a Borel set, and the measure $\mu_{n_j}$ is Radon; 
hence $\mu_{n_j}(K\cup K_j)>1-\ep$ for some compact $K_j\subset K_j'$.
\end{proof}

\begin{proposition}\label{p:example:3}
Suppose that $p\in\oms$, $X=\bt\om\setminus \sset p$, and $(\mu_n)_n\subset P(\om)$. Then there exists 
an infinite set $N\subset\N$ such that $\clx {P(X)}{\set{\mu_n:n\in N}}$ is compact.
\end{proposition}

\begin{proof}
Lemma~\ref{l:example:1}(2) implies the existence of sequences
$(N_k)_k$ and $(S_k)_k$ of infinite sets such that $N_{k+1}\subset N_k\subset\om$, $S_{k+1}\subset S_k\in p$, 
and $\mu_n(S_k)\leq \frac 1k$ for $k\in\N$ and $n\in N_k$. Take a sequence $(n_k)_k$ such that $n_k\in N_k$. 
We have $\mu_{n_k}(S_k)\leq \frac 1k$. Therefore, $\mu_{n_j} \notin W(U(S_k),\frac 1k)$ for $k\in \N$ and 
$j\geq k$. The set $F_k=P(\bt\om)\setminus W(U(S_k),\frac 1k)$ is closed in $P(\bt\om)$ and the set $\{j\in\N: \mu_{n_j}\notin F_k\}\subset\{1,2,...,k-1\}$ is finite. Consequently, $F_k$ contains the limit points $L$ of the sequence of measures $(\mu_{n_j})_j$.
The set $F=\bigcap_{k=1}^\infty F_k \subset P(\bt\om)$ is compact and $L\subset F$.

Let us show that $F\subset P(X)$. Take $\mu\in F$. Since $\mu\in F_k$ and $U(S_k)=\clx{\bt\om}{S_k}$, we have 
$\mu(\clx{\bt\om}{S_k})\leq \frac 1k$ for $k\in\N$. Hence $\mu(\{p\})=0$ and, consequently, $\mu\in P(X)$.

We set $N=\{n_j:j\in \N\}$ and $M=\set{\mu_n:n\in N}$. Since $\clx {P(\bt\om)}{M}= 
M\cup L\subset M\cup F\subset P(X)$, it follows that $\clx {P(X)}{\set{\mu_n:n\in N}}$ is compact.
\end{proof}

\begin{proposition}\label{p:example:4}
If $X=\bt\om\setminus \sset p$, where $p\in\oms$, then $P(X)$ is almost pseudo-$\om$-bounded.
\end{proposition}

\begin{proof}
Let $(U_n)_n$ be a sequence of nonempty open subsets of $P(X)$.
Since $P(\om)$ is dense in $P(\bt\om)$, there exists a $\mu_n\in U_n\cap P(\om)$.
Proposition~\ref{p:example:3} implies that $K=\cl {\set{\mu_n:n\in N}}$ is compact
for some infinite $N\subset \om$. Therefore,  the set $\set{n\in\N: U_n\cap K\neq\es}\supset N$ is infinite.
\end{proof}

\begin{proposition}\label{p:example:5}
Let $p\in\oms$ and $X=\bt\om\setminus\sset p$.
If $Y=\oms\setminus\sset p$ is \cccc/-bounded, then $P(X)$ is totally countably compact.
\end{proposition}

\begin{proof}
The space $P(Y)$ is $\om$-bounded by Proposition \ref{p:example:1}.
Let $(\ka_n)_n\subset P(X)$. Then $\ka_n=\la_n\mu_n+(1-\la_n)\nu_n$, where $\mu_n\in P(\om)$, $\nu_n\in P(Y)$, 
and $\la_n\in[0,1]$. Since $P(Y)$ is $\om$-bounded,  it follows that  $Q=\clx{P(Y)}{\set{\nu_n: n\in\N}}$ is 
compact. Proposition~\ref{p:example:3} implies that $S=\clx{P(X)}{\set{\mu_n: n\in N}}$ is compact for some 
infinite $N\subset \om$. Since
\[
K=\set{t\mu+(1-t)\nu: \mu\in S,\nu\in Q, t\in[0,1]}\subset P(X)
\]
is compact and $(\ka_n)_{n\in N}\subset K$,  it follows that $\clx{P(X)}{\set{\ka_n:n\in N}}$ is compact.
\end{proof}

Propositions~\ref{p:example:1}, \ref{p:example:2}, and~\ref{p:example:5} imply the following assertion.

\begin{proposition}\label{p:example:6}
There exists a $p\in\oms$ for which the following conditions hold:
\begin{enumerate}
\item
the space $\oms\setminus\sset p$ is \cccc/-bounded;
\item
$P(\oms\setminus\sset p)$ is $\om$-bounded and $P(\oms\setminus\sset p)=P_\be(\oms\setminus\sset p)$;
\item
$P(\bt\om\setminus \sset p)$ is totally countably compact.
\end{enumerate}
\end{proposition}

\begin{proposition}\label{p:example:7}
Let $X$ be any of the following spaces:
\begin{enumerate}
\item
the space $\om_1$ of all countable ordinals;
\item
the Tychonoff plane $\Pi_T=((\om_1+1)\times (\om+1))\setminus \sset{(\om_1,\om)}$;
\item
$X_{A,p}=[0.1]^A\setminus \sset p$, where $A$ is uncountable and $p\in [0,1]^A$.
\end{enumerate}
Then the spaces $X$ and $P(X)$ are pseudo-$\om$-bounded.
\end{proposition}

\begin{proof}
By Theorem~\ref{t:intro:other:3}, it suffices to show that $X$ is pseudo-$\om$-bounded.

Item (1) follows from Proposition~\ref{p:example:om1}.

(2) The space $\om_1$ is $\om$-bounded. Consequently, $(\om_1+1)\times \om$ is $\om$-bounded.
Since $(\om_1+1)\times \om$ is dense in $\Pi_T$,  it follows that  $\Pi_T$ is pseudo-$\om$-bounded.

(3) Let $p=(p_\al)_{\al\in A}$. Take $q=(q_\al)_{\al\in A}\in [0,1]^A$ such that $p_\al\neq q_\al$ for $\al\in A$. 
Let 
\[
\Si=\set{(x_\al)_{\al\in A}\in [0,1]^A: |\set{\al\in A: x_\al\neq q_\al}|\leq\om}.
\]
Then $\Si$ is dense in $X=[0,1]^A\setminus \sset p$ and $\om$-bounded. Therefore, $X$ is pseudo-$\om$-bounded.
\end{proof}

%%% end file('tex/sec/example.tex') %%%

\section{Mr\'owka Spaces}\label{sec:Mrowka}
%%% begin file('tex/sec/Mr\' owka.tex') %%%
Let $D$ be a countable set, $\cM=\set{M_\al: \al\in A}$ be a family of countable subsets of $D$, and 
$D\cap A=\es$. A family $\cM$ is called \term{almost disjoint} if $|M_\al\cap M_\be|<\om$ for distinct 
$\al,\be\in A$. 
%We denote $\ex M_\al=\sset \al \cup M_\al$.
We set $\Psi(D,\cM)=D\cup A$ and endow $\Psi(D,\cM)$ with the topology generated by the subbase
\[
\set{\sset u, \Psi(D,\cM)\setminus \sset u: u\in D} \cup \set{\sset\al\cup M_\al : \al\in A}.
\]
Spaces of the form $\Psi(D,\cM)$ are called \term{Mr\'owka--Isbell spaces} or \term{$\Psi$-spaces}.
The set $D$ is dense in $\Psi(D,\cM)$ and consists of isolated points. The set $A$ is discrete and closed in 
$\Psi(D,\cM)$. The space $\Psi(D,\cM)$ is locally compact. For $\al\in A$, the set $\sset\al\cup M_\al$ is 
a compact neighborhood of $\al$ and $M_\al$ is a sequence converging to~$\al$.

A family $\cM$ is said to be \term{maximal almost disjoint} if $\cM$ is an almost disjoint family and
for every infinite $M\subset D$, there exists an $\al\in A$ such that $|M\cap M_\al|=\om$.
A space $\Psi(D,\cM)$ is pseudocompact if and only if $\cM$ is a maximal almost disjoint family 
(see \cite{Hernandez-HernandezHrusak2018}, \cite[Section~11]{vanDouwen1984handbook}).

A family $\cM$ is called a \term{Mr\'owka family} if the Stone--\v Cech remainder 
$\bt \Psi(D,\cM) \setminus \Psi(D,\cM)$ is a singleton.
Mr\'owka--Isbell spaces $\Psi(D,\cM)$ with one-point Stone--\v Cech remainder are called \term{Mr\'owka spaces}.

\begin{note}
Many authors consider only $D=\om$  
\see\cite{Hernandez-HernandezHrusak2018} and non-indexed families $\cM$ (see \cite{Hernandez-HernandezHrusak2018}
and \cite[Section~11]{vanDouwen1984handbook}). However, for constructing a Mr\'owka family, it is more 
convenient to use indexed families.
\end{note}

We first recall the original construction of the Mr\'owka family and then describe a modification necessary 
for constructing a Mr\'owka space for which the space of probability measures is not pseudocompact.

We set $\nb=\sset{0,1}$.
As the set $D$ of isolated points, we use the set
\[
\ct = \bigcup_{n=0}^\infty \nb^n.
\]
The set $\ct$ is a Cantor tree and consists of finite binary sequences.
The zeroth power $\nb^0$ consists of the empty set, which is considered a sequence of length zero.
We will consider the set $\cs$ of infinite binary sequences with the product topology
of discrete two-point spaces $\nb$. The space $\cs$ is homeomorphic to the Cantor set.

For $x=(x_i)_i\in \cs$, we set
\[
C_x = \set{ x\rst n : n\in\om },\qquad \text{where }\ x\rst n = (x_0,x_1,\dots,x_{n-1}).
\]

\begin{proposition}[{\cite{Mrowka1977}, \cite[Section 8.6]{Hernandez-HernandezHrusak2018}}]
\label{p:Mrowka:1}
There exists a family of sets 
$\set{D_x:x\in \cs}$
satisfying the condition
\begin{itemize}
\item[\cms 1] $C_x\subset D_x \subset \ct$ for $x\in\cs$ and $\set{D_x:x\in \cs}$ is 
a maximal almost joint family.
\end{itemize}
\end{proposition}

\begin{proposition}[{\cite{Mrowka1977}, \cite[Section 8.6]{Hernandez-HernandezHrusak2018}}]\label{p:Mrowka:2}
There exists a family $\cP$ of sets satisfying the condition 
\begin{itemize}
\item[\cms 2]
$|P|=2$ for $P\in\cP$, the family $\cP$ is a partition of the space $\cs$, 
and for any disjoint closed uncountable sets $F_1,F_2\subset \cs$
there exists $P\in\cP$ such that $P\cap F_1\neq\es$ and $P\cap F_2\neq\es$.
\end{itemize}
\end{proposition}

\begin{proposition}[{\cite{Mrowka1977}, \cite[Section 8.6]{Hernandez-HernandezHrusak2018}}]\label{p:Mrowka:3}
Suppose given a family 
$\set{D_x:x\in\cs}$
satisfying condition \cms 1
and a family $\cP$ satisfying condition \cms 2.
Let $E_P=\bigcup_{x\in P} D_x$ for $P\in\cP$. Then $\cE=\set{E_P:P\in\cP}$ is a Mr\'owka family
and the space $\Psi(\ct,\cE)$ is a Mr\'owka space.
\end{proposition}

Propositions \ref{p:Mrowka:1}, \ref{p:Mrowka:2}, and \ref{p:Mrowka:3} imply the existence of a Mr\'owka space.
To construct the required Mr\'owka space, we need the following modification of Proposition~\ref{p:Mrowka:1}.

\begin{proposition}\label{p:Mrowka:4}
There exists a family of sets $\set{D_x:x\in \cs}$
which satisfies condition~\cms 1 and the following condition:
\begin{itemize}
\item[\cms 3] $|D_x\cap \nb^n|\leq 2$ for each $x\in\cs$ and $n\in \om$.
\end{itemize}
\end{proposition}

\begin{proof}
Let $\cL=\set{L\subset \ct: |L \cap \nb^n|\leq 1\text{ for }n\in\om}$. Then $\cC=\set{C_x:x\in\cs}\subset\cL$. 
We set 
\[
\gM=\set{\cM\subset \cL: \cC\subset\cM\text{ and }\cM\text{ is an almost disjoint family}}.
\]
Since the union of any chain in $\gM$ belongs to $\gM$, it follows by Zorn's lemma that $\gM$ contains 
an inclusion maximal element~$\cM$.

Let us show that $\cM$ is a maximal almost disjoint family. Take an indinite $S\subset \ct$.
There exists a set $Q\subset S$ such that $|Q\cap \nb^n|=1$ if $S\cap \nb^n\neq\es$. We have $Q\in \cL$.
Since $\cM$ is maximal in $\gM$,  it follows that  $Q\cap M$ is infinite for some $M\in\cM$. 
Therefore, $S\cap M$ is infinite.

We set $\cR=\cM\setminus \cC$. Since $|\cR|\leq 2^\om$, we have $|\cR|=|A|$ for some $A\subset \cs$.
We index the family $\cR$ by the elements of $A$: $\cR=\set{R_x: x\in A}$. For $x\in \cs$, we set
\[
D_x=\begin{cases}
C_x\cup R_x& \text{if }x\in A,
\\
C_x& \text{if }x\in \cs\setminus A.
\end{cases}
\]
The family $\set{D_x:x\in \cs}$ is as desired.
\end{proof}

\begin{theorem}\label{t:Mrowka:1}
There exists a Mr\'owka space $X$ for which $P(X)$ is not pseudocompact.
\end{theorem}

\begin{proof}
Let $\set{D_x:x\in \cs}$ be a family of sets satisfying conditions \cms 1 and \cms 3 
(see Proposition~\ref{p:Mrowka:4}), and let $\cP$ be a family of sets satisfying condition \cms 2 (see 
Proposition~\ref{p:Mrowka:2}).
We set $E_P=\bigcup_{x\in P} D_x$ for $P\in\cP$. By Proposition~\ref{p:Mrowka:3},
the family $\cE=\set{E_P:P\in\cP}$ is a Mr\'owka family and the space $X=\Psi(\ct,\cE)$ is a Mr\'owka space.
It follows from \cms 2 and \cms 3 that 
\begin{itemize}
\item[\cms 4] $|E_P\cap \nb^n|\leq 4$ for each $P\in\cP$ and $n\in \om$.
\end{itemize}

\vskip 0.3em

Let us show that $P(X)$ is not pseudocompact.
For $n\in\om$, we set
\begin{equation}\label{eq:Mrowka:un}
U_n=\set{\nu\in P(X): \nu(\sset u)> \tfrac 1{2^n}(1-\tfrac 1{2^n})\text{ for }u\in\nb^n}.
\end{equation}
Since $U_n=\bigcap\set{W(\sset u,\tfrac 1{2^n}(1-\tfrac 1{2^n})): u\in\nb^n}$, 
it follows that $U_n$ is open in $P(X)$.
Let us show that the sequence $(U_n)_n$ does not accumulate to any point in $P(X)$. Assume that, on the contrary, 
$(U_n)_n$ accumulates to some $\mu\in P(X)$. Since $X$ is a scattered space,
we have $\mu(\sset u)=c>0$ for some $u\in X$. Consider two cases.

The first case: $u\in\ct$. In this case, $u\in\nb^m$ for some $m\in\om$. Let $U=W(\sset u,\frac c2)$.
The set $U$ is an open neighborhood of the measure $\mu$. Take $N\in\N$ such that $\frac 1{2^N}<\frac c2$ 
and $N>m$. Let $n>N$, and let $\nu\in U_n$. Since $\nu(\nb^n)>1-\frac 1{2^n}$, we have 
$\nu(\sset u)<\frac 1{2^n}<\frac 1{2^N}<\frac c2$.
Consequently, $U_n\cap U=\es$ for $n>N$. This contradicts the assumption that $(U_n)_n$ accumulates to~$\mu$.

The second case: $u\in\cP$. Let $U=W(S,\frac c2)$, where $S=\sset u\cup E_u$. The set $U$ is an open 
neighborhood of the measure $\mu$.
Take $N\in\N$ such that
\begin{equation}\label{eq:Mrowka:1}
(2^n-4)\tfrac 1{2^n}(1-\tfrac 1{2^n}) > 1- \tfrac c2
\end{equation}
for all $n>N$. Let $n>N$ and $\nu\in U_n$. From \cms 4 it follows that $|\nb^n\cap S|\leq 4$,
and from (\ref{eq:Mrowka:un}) and (\ref{eq:Mrowka:1}) it follows that $\nu(\nb^n\setminus S)>1- \tfrac c2$.
Consequently, $\nu(S)< \tfrac c2$ and $\nu\notin U$. This means that $U_n\cap U=\es$ for $n>N$.
This contradicts the assumption that $(U_n)_n$ accumulates to~$\mu$.
\end{proof}

%%% end file('tex/sec/Mr\' owka.tex') %%%

\section{Factorization Theorems}\label{sec:ft}
%%% begin file('tex/sec/ft.tex') %%%
Let $n\in \N$, and let $X_1$, $X_2$, \dots, $X_n$ be spaces. We say that a product $\prod_{i=1}^n X_i$ is 
\term{$\R$-factorizable} if, given any continuous function $f\: \prod_{i=1}^n X_i\to\R$, there exist 
separable metrizable spaces $Y_i$ and continuous mappings $g_i\: X_i\to Y_i$ for $i=1,2,\dots,n$ and a continuous
function $h\: \prod_{i=1}^n Y_i\to \R$ such that $f=h \circ \prod_{i=1}^n g_i$.
The definition of \term{weak} \term{$\R$-factorizability} is obtained by removing the continuity 
requirement on~$h$.

\begin{note}
The concept of $\R$-factorizable products arose in the study of $\R$-factorizable groups
\see\cite{ReznichenkoSipacheva2013,Reznichenko2024rf}. A topological group $G$ is said to be 
\term{$\R$-factorizable} if, for any continuous function $f\: G\to \R$, there exist a separable metrizable 
topological group $G$, a continuous homomorphism $g\: G\to H$, and a continuous function $h\: H\to \R$ 
such that $f=h\circ g$. If a product $G\times H$
of topological groups is $\R$-factorizable as a group, then $G\times H$ is $\R$-factorizable
as a product \see\cite{ReznichenkoSipacheva2013,Reznichenko2024rf}.
Chapter~8 of the monograph \cite{at2009} considers several multiplicative classes of $\R$-factorizable groups,
which gives additional examples of $\R$-factorizable products. The class of all $\R$-factorizable groups
is not multiplicative \see\cite{Reznichenko2024rf,Sipacheva2023-2}. In \cite{Reznichenko2024efualg} \ 
$\R$-factorizable products were used to study topological universal algebras.
\end{note}

\begin{assertion}\label{a:ft:d1}
A product $\prod_{i=1}^n X_i$ is $\R$-factorizable if and only if, given any separable metrizable
space $Z$ and any continuous mapping $f\: \prod_{i=1}^n X_i\to Z$, there exist separable metrizable
spaces $Y_i$ and continuous mappings $g_i\: X_i\to Y_i$ for $i=1,2,\dots,n$ and 
a continuous mapping $h\: \prod_{i=1}^n Y_i\to Z$ such that $f=h \circ \prod_{i=1}^n g_i$.
A characterization of weak $\R$-factorizability is obtained by removing the continuity requirement on~$h$.
\end{assertion}

\begin{proof}
The implication $(\larr)$ is obvious. Let us prove $(\rarr)$. We will assume that $Z\subset \R^\N$. 
Let $f_m=\pi_m\circ f$, where $\pi_m$ is the projection of $\R^\N$ onto the $m$th factor. Since 
$\prod_{i=1}^n X_i$ is (weakly) $\R$-factorizable,  it follows that 
there exist separable metrizable spaces $Y_{m,i}$ and continuous mappings $g_{m,i}\: X_i\to Y_{m,i}$ 
for $i=1,2,\dots,n$ and 
a continuous (not necessarily continuous) function $h_m\: \prod_{i=1}^n Y_{m,i}\to \R$ 
such that $f_m=h_m \circ \prod_{i=1}^n g_{m,i}$.
We set $g_i=\diag_{m=1}^\infty g_{m,i}$ and $Y_i=g_i(X_i)$ and let $p_{m,i}: Y_i\to Y_{m,i}$ be the projection. 
We also set 
\[{
\def\({\left(\vphantom{\prod}\right.}
\def\){\left.\vphantom{\prod}\right)}
h=\prod_{m=1}^\infty \( h_m \circ \prod_{i=1}^n p_{m,i}\): \prod_{i=1}^n Y_i \to Z\subset \R^\N.
}\]
Then $f=h \circ \prod_{i=1}^n g_i$.
\end{proof}

\begin{assertion}\label{a:ft:d2}
Let $X_0$ be a metrizable compact set. The product $\prod_{i=0}^n X_i$ is $\R$-factorizable
if and only if so is the product $\prod_{i=1}^n X_i$. A similar assertion holds for weak $\R$-factorizability.
\end{assertion}
\begin{proof}
The implication $(\rarr)$ is obvious. We prove $(\larr)$.
Let $f\: \prod_{i=0}^n X_i\to \R$ be a continuous function.
We denote by $C_b(X_0)$ the space of continuous functions on $X_0$ with the topology of uniform convergence. Let
\[
\ph: \prod_{i=1}^n X_i\to C_u(X_0),\ \ph((x_1,\dots,x_n))(x_0) = f((x_0,x_1,\dots,x_n)).
\]
The mapping $\ph$ is continuous and $C_u(X_0)$ is metrizable and separable. Proposition~\ref{a:ft:d1} 
implies the existence of separable metrizable spaces $Y_i$ and continuous mappings $g_i\: X_i\to Y_i$ for 
$i=1,2,\dots, n$ and 
a continuous (not necessarily continuous) mapping $\psi\: \prod_{i=1}^n Y_i\to C_u(X_0)$ such that 
$\ph=h \circ \prod_{i=1}^n g_i$. We set $Y_0=X_0$,  $g_0=\id_{X_0}$, and 
\[
h: \prod_{i=0}^n Y_i\to \R,\ (y_0,y_1,...,y_n) \mapsto \psi((y_1,...,y_n))(x_0).
\]
Then $f=h \circ \prod_{i=0}^n g_i$.
\end{proof}

\begin{assertion}[{\cite[Proposition 12]{Reznichenko2024efualg}}]\label{a:ft:1}
If $\prod_{i=1}^n X_i$ is Lindel\"of, then the product $\prod_{i=1}^n X_i$ is $\R$-factorizable.
\end{assertion}

\begin{cor}\label{c:ft:1}
If $X_i$ are compact for $i=1,\dots,n$, then the product $\prod_{i=1}^n X_i$ is $\R$-factorizable.
\end{cor}

\begin{assertion}\label{a:ft:2}
If $\prod_{i=1}^n X_i$ is pseudocompact, then the product $\prod_{i=1}^n X_i$ is $\R$-factorizable.
\end{assertion}

\begin{proof}
Let $f\: \prod_{i=1}^n X_i\to\R$ be a continuous function.
From Gliksberg's theorem \see\cite{gli1959} it follows that $\bt(\prod_{i=1}^n X_i)=\prod_{i=1}^n \bt X_i$.
Consequently, the function $f$ extends to a continuous function $\hat f: \prod_{i=1}^n \bt X_i \to \R$.
By Corollary~\ref{c:ft:1} there exist separable metrizable spaces $Y_i$
and continuous mappings $\hat g_i: \bt X_i\to Y_i$ for $i=1,2,\dots, n$ and 
a continuous function $h\: \prod_{i=1}^n Y_i\to \R$ such that $\hat f=h \circ \prod_{i=1}^n \hat g_i$.
Let $g_i = \restr{\hat g_i}{Y_i}$ for $i=1,\dots,n$. Then $f=h \circ \prod_{i=1}^n g_i$.
\end{proof}

\begin{assertion}\label{a:ft:3}
If $X_i$ is pseudocompact and $\om_1$ is a precaliber of $X_i$ for $i=1,2,\dots,n$, then the product 
$\prod_{i=1}^n X_i$ is weakly $\R$-factorizable.
\end{assertion}

\begin{proof}
Let $f\: \prod_{i=1}^n X_i\to\R$ be a continuous function.
From \cite[Theorem 4]{Reznichenko2022} it follows that the function $f$ extends to a separately
continuous function $\hat f\: \prod_{i=1}^n \bt X_i\to \R$. 
According to \cite[Theorem 1.10]{ReznichenkoUspenskij1998},
there exist separable metrizable spaces $Y_i$ and continuous mappings $\hat g_i\: \bt X_i\to Y_i$ 
for $i=1,2,\dots,n$ and
a separately continuous function $h\: \prod_{i=1}^n Y_i\to \R$ such that $\hat f=h \circ \prod_{i=1}^n \hat g_i$.
Let $g_i = \restr{\hat g_i}{Y_i}$ for $i=1,\dots,n$. Then $f=h \circ \prod_{i=1}^n g_i$.
\end{proof}

Consider the $(n-1)$-dimensional simplex
\[
S_n=\set{(\la_i)_i\in \R^{n}: \sum_{i=1}^{n}\la_i=1,\ x_i\geq 0\text{ for }i=1,2,...,n}.
\]
The space $S_n$ is metrizable and compact. Given a space $X$, we introduce the mapping
\[
\ph^X_n: S_n \times X^n \to P_n(X),\ ((\la_1,\la_2,...,\la_n),x_1,x_2,...,x_n) \mapsto \sum_{i=1}^n \la_i \de_{x_i}.
\]
The mapping $\ph^X_n$ is continuous.

\begin{proposition}\label{p:ft:1}
Suppose given $n\in\N$, a space $X$ for which $X^n$ is weakly $\R$-factorizable, and a continuous 
function $f\: P_n(X)\to \R$.
Then there exists a separable metrizable space $Y$, a continuous surjective mapping $g\: X\to Y$, and 
a mapping $h\: P_n(Y)\to\R$ such that $f=h \circ P(g)$.
\end{proposition}

\begin{proof}
Proposition \ref{a:ft:d2} implies that $S_n\times X^n$ is weakly $\R$-factorizable.
We set $\vt f=f\circ \ph^X_n\: S_n \times X^n\to\R$,  $X_0=S_n$, and $X_i=X$ for $i=1,2,\dots,n$.
Since the product $\prod_{i=0}^n X_i$ is weakly $\R$-factorizable,  it follows that  there exist 
separable metrizable spaces $Y_i$ and 
continuous mappings $g_i\: X_i\to Y_i$ for $i=0,1,\dots,n$  and a function $s\: \prod_{i=0}^n Y_i\to \R$ 
such that $\vt f=s \circ \prod_{i=0}^n g_i$. We set 
\[
g=\diag_{i=1}^n g_i: X\to \prod_{i=1}^n Y_i, 
\]
$Y=g(X)$,  and $\pi_0=g_0$  and let $\pi_i\: Y\to Y_i$ be the projections for $i=1,2,\dots,n$.

Let $\vt h = s \circ \prod_{i=0}^n \pi_i: S_n\times Y^n\to \R$. Then $\vt f=\vt h\circ (\id_{S_n}\times g^n)$. 
Hence 
\begin{equation}\label{eq:ft:1}
{
\def\({\left(\vphantom\sum\right.}
\def\){\left.\vphantom\sum\right)}
f\(\sum_{i=1}^n\la_i\de_{x_i}\) = f\(\sum_{i=1}^n\la_i\de_{u_i}\)
\text{ if }g(x_i)=g(u_i)
}
\end{equation}
for $i=1,\dots,n$ and $(\la_1,\dots,\la_n)\in S_n$.
Let $\xi\: Y\to X$ be a mapping such that $g(\xi(y))=y$ for $y\in Y$. We set 
{
\def\({\left(\vphantom\sum\right.}
\def\){\left.\vphantom\sum\right)}
\[
h\(\sum_{i=1}^n\la_i\de_{y_i}\) = f\(\sum_{i=1}^n\la_i\de_{\xi(y_i)}\)
\]
}
for $\sum_{i=1}^n\la_i\de_{y_i}\in P_n(Y)$.

Let us show that $f=h \circ P(g)$. Take $\mu=\sum_{i=1}^n\la_i\de_{x_i}\in P_n(X)$. Then
$P(g)(\mu)=\sum_{i=1}^n\la_i\de_{g(x_i)}$ and $h(P(g)(\mu))=f(\sum_{i=1}^n\la_i\de_{u_i})$,
where $u_i=\xi(g(x_i))$ for $i=1,2,\dots,n$. Since $g(x_i)=g(u_i)$, it follows from (\ref{eq:ft:1}) 
that $f(\mu)=h(P(g)(\mu))$.
\end{proof}

\begin{theorem}\label{t:ft:1}
Let $X$ be a space such that $X^n$ are weakly $\R$-factorizable for all $n\in \N$, and let 
$f\: P_f(X)\to \R$ be a continuous function.
Then there exists a separable metrizable space $Y$, a continuous surjective mapping $g\: X\to Y$, and
a mapping $h\: P_f(Y)\to\R$ such that $f=h \circ P(g)$.
\end{theorem}

\begin{proof}
We set $f_n=\restr f{P_n(X)}$ for $n\in\N$. Since $X^n$ is weakly $\R$-factorizable, it follows by 
Proposition~\ref{p:ft:1} that there exists a separable metrizable space $Y_n$, a continuous surjective mapping 
$g_n\: X\to Y_n$, and a mapping $h_n\: P_n(Y_n)\to\R$ such that $f_n=h_n \circ P(g_n)$.
We set $g=\diag_{n=1}^\infty g_n$ and $Y=g(X)$ and let $\pi_n\: Y\to Y_n$ be the projection. 
For $\mu\in P_n(Y)\subset P_f(Y)$, we set $h(\mu)=h_n(P(\pi_n)(\mu))$.
The mapping $h\: P_f(Y)\to \R$ is well defined in view of the equality $f_n=h_n\circ P(g_n)$, and 
$f=h \circ P(g)$.
\end{proof}

The following proposition is easy to verify.

\begin{proposition}\label{p:ft:1+1}
For any sets  $X$, $Y$, and $Z$, mapping $f\: X\to Z$, and surjective mapping $g\: X\to Y$,
the following conditions are equivalent:
\begin{enumerate}
\item
there exists an $h\: Y\to Z$ such that $f=h\circ g$;
\item
$|f(g^{-1}(y))|=1$ for $y\in Y$;
\item
if $u,v\in X$ and $g(u)=g(v)$, then $f(u)=f(v)$.
\end{enumerate}
\end{proposition}

Let $g\: X\to Y$ be a continuous mapping, and let $S\subset X$. We say that $S$ is \term{$g$-dense}
or \term{dense with respect to the mapping $g$} if, for any open $U,V\subset X$ 
such that $g(U)\cap g(V)\neq\es$, we have $g(U\cap S)\cap g(V\cap S)\neq\es$.

\begin{proposition}\label{p:ft:2}
Suppose given a continuous mapping $g\: X\to Y$, a  $g$-dense set  $S\subset X$, and a continuous 
function  $f\: X\to\R$. If there exists a function $h_S\: g(S)\to \R$ such that $\restr fS= h_S\circ \restr gS$,
then $h_S$ extends to a function $h\: Y\to \R$ such that $f=h\circ g$.
\end{proposition}

\begin{proof}
In view of Proposition~\ref{p:ft:1+1}, it suffices to check that if $u,v\in X$ and $g(u)=g(v)$, then $f(u)=f(v)$. 
Assume that, on the contrary, $f(u)\neq f(v)$. Then $f(U)\cap f(V)=\es$ for some neighborhoods $U$ and $V$ 
of the points $u$ and $v$.
Since $g(u)=g(v)\in g(U)\cap g(V)\neq\es$, it follows from the $g$-density of $S$ 
that $g(\vt u)=g(\vt v)$ for some $\vt u\in U\cap S$ and $\vt v\in V\cap S$. We have 
\[
f(\vt u)=h_S(g(\vt u))=h_S(g(\vt v))=f(\vt v).
\]
Therefore, $f(\vt u)=f(\vt v)\in f(U)\cap f(V)$.
This contradicts the assumption $f(U)\cap f(V)=\es$. 
\end{proof}

\begin{proposition}\label{p:ft:3}
Let $g\: X\to Y$ be a continuous surjective map of compact spaces.
Then the set $P_f(X)$ is dense with respect to the map $P(g)\: P(X)\to P(Y)$.
\end{proposition}

\begin{proof}
Let $U,V\subset P(X)$ be open sets for which $P(g)(U) \cap P(g)(V)\neq\es$. We must show that
\[
P(g)(P_f(X)\cap U) \cap P(g)(P_f(X)\cap V)\neq\es.
\]
Take $\mu\in U$ and $\nu\in V$ for which $P(g)(\mu)=P(g)(\nu)$. There exist $m,n\in\N$,
functionally open sets $U_1$, \dots, $U_m$, $V_1$, \dots, $V_n$, and positive numbers $a_1$, \dots, $a_m$, 
$b_1$, \dots, $b_n$ such that
\begin{align*}
\mu&\in W((U_i,a_i)_{i=1}^m)\subset U,
&
\nu&\in W((V_i,b_i)_{i=1}^n)\subset V.
\end{align*}
We set $\cU=\set{U_i: i=1,\dots,m}$, $\cV=\set{V_i: i=1,\dots,n}$,
\begin{align*}
\cS &= \prt{\cU\cup \cV},
&
\cP &= \prt{\set{ g(S): S\in\cS}},
\\
\cR &= \set{ g^{-1}(P): P\in\cP},
&
\cQ &= \prt{\cS\cup \cR}.
\end{align*}
The partitions $\cS$, $\cP$, $\cR$, and $\cQ$ are finite, the partition $\cS$ consists of $\si$-compact sets,
and the partitions $\cP$, $\cR$, and $\cQ$ consist of Borel sets. Since $\cS$ and $\cR$ are partitions, we have
\[
\cQ = \set{S\cap R: S\in\cS,\ R\in \cR}.
\]
For $P\in\cP$, we set
\[
\cQ_P = \set{Q\in \cQ: g(Q)\subset P}=\set{S\cap g^{-1}(P): S\in\cS,\ S\cap g^{-1}(P)\neq\es}.
\]
Since $g(S)=\bigcup\set{P\in\cP: P\subset g(S)}$ for $S\in\cS$,  it follows that 
\[
\cQ_P = \set{S\cap g^{-1}(P): S\in\cS,\ P\subset g(S)}.
\]
Therefore, $g(Q)=P$ for $Q\in \cQ_P$. For $P\in\cP$ and $Q\in\cQ_P$, we choose $y_P\in P$
and $x_Q\in Q$ so that $g(x_Q)=y_P$. We set 
\begin{align*}
\vt \mu &= \sum_{Q\in \cQ} \mu(Q)\de_{x_Q},
&
\vt \nu &= \sum_{Q\in \cQ} \nu(Q)\de_{x_Q}.
\end{align*}
Clearly, $\vt \mu, \vt \nu\in P_f(X)$. By Proposition~\ref{a:tms:1}, we have 
\begin{align*}
\vt\mu&\in W((U_i,a_i)_{i=1}^m)\subset U,
&
\vt\nu&\in W((V_i,b_i)_{i=1}^n)\subset V.
\end{align*}
Since $P(g)(\mu)=P(g)(\nu)$,  it follows that $\mu(g^{-1}(P))=\nu(g^{-1}(P))$ for $P\in\cP$.
Hence
\[
P(g)(\vt\mu) = \sum_{P\in \cP} \mu(g^{-1}(P))\de_{y_P} = \sum_{P\in \cP} \nu(g^{-1}(P))\de_{y_P} = P(g)(\vt\nu).
\]
\end{proof}

\begin{theorem}\label{t:ft:0}
Let $X$ be a compact space, and let $R\subset P(X)$ be a Lindel\"of subspace.
If $f\: R\to \R$ is a continuous function, then there exists a metrizable compact space $Y$,
a continuous surjective mapping $g\: X\to Y$, and a continuous mapping $h\: P(g)(R)\to\R$ such that
\begin{align}\label{eq:ft:0:t}
f=h \circ \restr{P(g)}R.
\end{align}
\end{theorem}

\begin{proof}
Since the family of continuous functions $\set{\hat q:q\in C^*(X)}$ on $P(X)$ generates the topology of $P(X)$, 
it follows that the diagonal product
\[
j=\diag_{q\in C^*(X)} \vh q: P(X) \to \R^{C^*(X)}
\]
is a topological embedding. We set 
\begin{align}\label{eq:ft:0:-2}
\wt R&=j(R), & \vt f&=f\circ \restr{j^{-1}}{\wt R}.
\end{align}
Then $\wt R\subset \R^{C^*(X)}$ is Lindel\"of, and Tkachenko's lemma \cite[Lemma~8.1.4]{at2009}
implies the existence of a countable set $\cA\subset C^*(X)$ and a continuous mapping 
$\vt h\: \pi_\cA(\wt R)\to \R$ such that
\begin{align}\label{eq:ft:0:-3}
\vt f&=\vt h \circ \restr{\pi_\cA}{\wt R},
\end{align}
where $\pi_\cA: \R^{C^*(X)} \to \R^\cA$ is the projection.

We set $g=\pi_\cA\circ j\circ \de^X=\diag\cA$ and $Y=g(X)$. For $q\in\cA$, let $p_q\:Y\to\R$, 
$(x_s)_{s\in\cA}\mapsto x_q$, be the projection onto the $q$th coordinate, and let 
\[
r\: P(Y) \to \R^\cA,\ \mu \mapsto (\vh p_q(\mu))_{q\in\cA}.
\]
Then
\begin{align}\label{eq:ft:0:-4}
r\circ P(g)&=\pi_\cA \circ j.
\end{align}
We set 
\begin{align}\label{eq:ft:0:-1}
h &= \vt h \circ \restr r{P(g)(R)}.
\end{align}
Let us check (\ref{eq:ft:0:t}).
Since (\ref{eq:ft:0:-2}) and (\ref{eq:ft:0:-3}),  it follows that 
\begin{align}\label{eq:ft:0:2}
f &= \vt h \circ \pi_\cA \circ \restr{j}{R}.
\end{align}
Relations (\ref{eq:ft:0:-1}) and (\ref{eq:ft:0:-4}) imply 
\begin{align}\label{eq:ft:0:3}
h \circ \restr{P(g)}R &=
\vt h \circ r \circ \restr{P(g)}R
= \vt h \circ \pi_\cA \circ \restr jR, 
\end{align}
and relations (\ref{eq:ft:0:2}) and (\ref{eq:ft:0:3}) imply~(\ref{eq:ft:0:t}).
\end{proof}

Theorem~\ref{t:ft:0} implies the following assertion for $R=P(X)$.

\begin{theorem}\label{t:ft:0+1}
Let $X$ be a compact space. Then, for any continuous function 
$f\: P(X)\to \R$, then there exists a metrizable compact space $Y$,
a continuous surjective mapping $g\: X\to Y$, and a continuous mapping $h\: P(Y)\to\R$ such that
\begin{align*}
f=h \circ P(g).
\end{align*}
\end{theorem}

\begin{proposition}\label{p:ft:5}
Suppose given  a compact space $X$,  a $G_\de$-dense set $S\subset X$, and a metrizable compact set $Y$.
Let $g\: X\to Y$ be a continuous surjective mapping.
Then the set $P_f(S)$ is dense with respect to the mapping $P(g)\: P(X)\to P(Y)$.
\end{proposition}

\begin{proof}
Let $U,V\subset P(X)$ be open sets. Suppose that $\mu\in U$, $\nu\in V$, and $P(g)(\mu)=P(g)(\nu)$.
It suffices to verify that $P(g)(\vt\mu)=P(g)(\vt\nu)$ for some $\vt\mu\in P_f(S)\cap U$ and 
$\vt\nu \in P_f(S)\cap V$.

If $U\cap V\neq\es$, then the density of $P_f(S)$ in $P(X)$ implies the existence of an 
$\eta\in P_f(S)\cap U\cap V$ for which $P(g)(\eta)=P(g)(\eta)$.
Suppose that $U\cap V=\es$. Let $f$ be a continuous function on $P(X)$ such that 
$\mu\in f^{-1}((-\infty,0))\subset U$
and $\nu\in f^{-1}((0,+\infty))\subset V$. Theorem~\ref{t:ft:0+1} implies the existence of a
metrizable compact space $Z_*$, a continuous surjective mapping $q_*\: X\to Z_*$, and a continuous mapping 
$h_*\: P(Z_*)\to\R$ such that $f=h_* \circ P(q_*)$.
We set $q=g\diag q_*\: X\to Y\times Z_*$, $Z=q(X)$,
\begin{align*}
p &\: Z\to Y,\ (y,z_*) \mapsto y,
\\
p_* &\: Z\to Z_*,\ (y,z_*) \mapsto z_*,
\end{align*}
and $h=h_* \circ P(p_*): P(Z)\to\R$. Then the compact set $Z$ is metrizable, $g=p \circ q$, 
and $f=h\circ P(q)$, i.e. the following diagram is commutative:
\[{
% A \arrow[r] \arrow[d] \arrow[dr, phantom, "q", very near start]
%\shorthandoff{"}
\begin{tikzcd}
&\R
\\
P(X) 
\arrow[rr,"P(q)"] 
\arrow[ru,"f"]
&& P(Z) 
\arrow[lu,"h"']
\\
X \arrow[rr,"q"] \arrow[rd,"g"'] \arrow[u,"\de_X"]
&& Z \arrow[ld,"p"] \arrow[u,"\de_Z"']
\\
&Y
\end{tikzcd}
%\shorthandon{"}
}\]
We set $\wt U=h^{-1}((-\infty,0))$ and $\wt V=h^{-1}((0,+\infty))$.
Then
\begin{align}\label{eq:ft:uv}
\mu &\in P(q)^{-1}(\wt U)\subset U,
&
\nu &\in P(q)^{-1}(\wt V)\subset V.
\end{align}

\begin{lemma}\label{l:ft:1}
$P(q)(P_f(S))=P_f(Z)$.
\end{lemma}

\begin{proof}
Let $\ka=\sum_{i=1}^n \la_i \de_{z_i}\in P_f(Z)$.
Note that the sets $q^{-1}(z_i)$ are of type $G_\de$ in $X$, because $Z$ is metrizable. 
Since $S$ is $G_\de$-dense in $X$,  it follows that  there exist $x_i\in S\cap q^{-1}(z_i)$.
We set $\vt \ka=\sum_{i=1}^n \la_i \de_{x_i}$. Then $P(q)(\vt\ka)=\ka$ and $\vt \ka\in P_f(S)$.
\end{proof}

By Proposition \ref{p:ft:3}, the set $P_f(Z)$ is $P(p)$-dense.
Since $P(g)(\mu)=P(p)(P(q)(\mu))\in P(p)(\wt U)\cap P(p)(\wt V)\neq\es$ and $P_f(Z)$ is $P(p)$-dense, 
it follows that 
$P(p)(\mu_*)=P(p)(\nu_*)$ for some $\mu_*\in \wt U\cap P_f(Z)$ and $\nu_*\in \wt V\cap P_f(Z)$.
By Lemma~\ref{l:ft:1},
there exists a $\vt \mu\in P(q)^{-1}(\mu_*)\cap P_f(S)$ and a $\vt \nu\in P(q)^{-1}(\nu_*)\cap P_f(S)$.
We have $\vt \mu,\vt \nu\in P_f(S)$ and 
\begin{gather*}
P(g)(\vt\mu)=P(p)(P(q)(\vt\mu))=P(p)(\mu_*)=P(p)(\nu_*)
\\
=P(p)(P(q)(\vt\nu))=P(g)(\vt\nu).
\end{gather*}
Since $P(q)(\vt\mu)=\mu_*\in \wt U$ and $P(q)(\vt\nu)=\nu_*\in \wt V$,  it follows from (\ref{eq:ft:uv}) 
that $\vt\mu\in U$ and $\vt\nu\in V$.
\end{proof}

\begin{theorem}\label{t:ft:2}
Let $X$ be a pseudocompact space such that $X^n$ are weakly $\R$-factorizable for all $n\in\N$ and 
$P_f(X)\subset R\subset P(\bt X)$. Then, for every continuous function $f\: R\to\R$, there exists 
a metrizable compact set $Y$, a
continuous surjective mapping $g\: X\to Y$, and a function $h\: P(Y)\to \R$ such that
\[
f=h\circ \restr {P(\bt g)}R.
\]
\end{theorem}

\begin{proof}
We set $\vt f=\restr f{P_f(X)}$.
Since $X^n$ is weakly $\R$-factorizable for every $n\in\N$, Theorem~\ref{t:ft:1} gives a
separable metrizable space $Y$, a continuous surjective mapping $g\: X\to Y$, and a mapping 
$\vt h\: P_f(Y)\to\R$ such that $\vt f=\vt h \circ \restr {P(g)}{P_f(X)}$. The space $Y$ is compact, because 
$X$ is pseudocompact.
It also follows that $X$ is a $G_\de$-dense subset of $\bt X$ (see \cite[Proposition~3.7.20]{at2009}).
By Proposition~\ref{p:ft:5}, the set $P_f(X)$ is dense
with respect to the mapping $P(\bt g)\: P(\bt X)\to P(Y)$.
Since $\restr{P(\bt g)}{P(X)}=P(g)$,  it follows by Proposition~\ref{p:ft:2} that 
the function $\vt h$ extends to a function $h\: P(Y)\to\R$ such that $f=h\circ \restr {P(\bt g)}R$.
\end{proof}

\begin{proposition}\label{p:ft:6}
Suppose given  a pseudocompact space $X$,  metrizable compact spaces $Y$ and $Z$,
a continuous surjective mapping $g\: X\to Y$, and a mapping $h\: Y\to Z$. If the composition 
$f=h\circ g\: X\to Z$ is continuous, then the mapping~$h$ is continuous.
\end{proposition}

\begin{proof}
Since the mapping $g\diag f$
is continuous, it follows that the set $\Gamma_h = (g\diag f) (X)$, being a continuous image of 
the pseudocompact space $X$ in the metrizable space $Y\times Z$, is compact and, therefore, closed in 
$Y\times Z$. The set
\[
\Gamma_h = \set{(g(x),h(g(x))):x\in X}=\set{(y,h(y)):y\in Y}
\]
is the graph of $h$, because $g$ is surjective. Therefore, $h$ is continuous, because its  graph 
is closed in $Y\times Z$ and the space $Z$ is compact \seee\cite[3.1.D]{EngelkingBookRu}.
\end{proof}

%%% end file('tex/sec/ft.tex') %%%

\section{Main Results}\label{sec:main}
%%% begin file('tex/sec/main.tex') %%%

\begin{theorem}\label{t:main:1}
Suppose that $X$ is a pseudocompact space such that $X^n$ are weakly $\R$-factorizable for all $n\in\N$ and 
$Y\subset P(\bt X)$ is pseudocompact, and $P_f(X)\subset Y$. Then $Y$ is $C^*$-embedded in $P(\bt X)$.
\end{theorem}

\begin{proof}
Let $f\: Y\to \R$ be a bounded continuous function. Theorem~\ref{t:ft:2} implies the existence of  
a metrizable compact set $K$, a continuous surjective mapping $g\: X\to K$, and a function $h\: P(K)\to \R$ 
such that $f=h\circ \restr {P(\bt g)}Y$. We set $\vt g=\restr {P(\bt g)}Y$. Then $f=h\circ \vt g$.

Let us show that the mapping $\vt g$ is surjective.
We have $P(g)(P_f(X))=P_f(K)$, because $g$ is surjective. 
Since $P_f(X)\subset Y$ and $P_f(K)$ are dense in $P(K)$, it follows that  $\vt g(Y)=P(\bt g)(Y)$ is dense 
in $P(K)$, and since $Y$ is pseudocompact,  it follows that  $\vt g(Y)$ is pseudocompact and 
dense in the metrizable compact set $P(K)$. Consequently, $\vt g(Y)$ is compact and coincides with~$P(K)$.

Since $f$ is bounded, $f=h\circ \vt g$, and the mapping $\vt g$ is surjective, we have $f(Y)=h(P(K))\subset 
[-a,a]=Z$ for some $a>0$. We assume that $f$ is a mapping from $Y$ to $Z$ and $h$ is a mapping from $P(K)$ to 
$Z$. Since $Y$ is pseudocompact, $Z$ and $P(K)$ are metrizable compact spaces, $\vt g\: Y\to P(K)$ is a 
continuous surjective mapping, $h\: P(K)\to Z$ is a mapping, and the composition $f=h\circ \vt g\: Y\to Z$ is 
continuous, it follows from Proposition~\ref{p:ft:6} that $h$ is continuous.

We set $\hat f=h\circ P(\bt g)$. The function $\hat f\: P(\bt Y)\to\R$ is continuous and extends~$f$.
\end{proof}

\begin{proof}[of Theorem \ref{t:intro:1}]
Proposition \ref{a:ft:2} implies that the product $X^n$ is $\R$-factorizable for all $n\in\N$.
By Theorem~\ref{t:main:1}, the space $Y$ is $C^*$-embedded in~$P(\bt X)$.
\end{proof}

\begin{proof}[of Theorem \ref{t:intro:4}]
Proposition~\ref{a:ft:3} implies that the product $X^n$ is weakly $\R$-factorizable for all $n\in\N$.
By Theorem~\ref{t:main:1}, the space $Y$ is $C^*$-embedded in~$P(\bt X)$.
\end{proof}

%ччхxxxx
\begin{proof}[of Theorem~\ref{t:intro:2}]
Assume the contrary. Let $n$ be the smallest integer for which $X^n$ is not pseudocompact.
There exists a locally finite sequence $(U_m)_m$ of nonempty open subsets of $X^n$ 
such that $U_m=\prod_{i=1}^n U_{m,i}$.
We set $\ep=\frac 1{2n}$ and
\begin{align*}
W_m &= W( (U_{m,i}, \ep)_{i=1}^n )
\end{align*}
for $m\in\N$.
We have $W_m\neq\es$, because $n\ep<1$.
Since $P(X)$ is pseudocompact,  it follows that  the sequence $(W_m)_m$ accumulates to some measure 
$\mu\in P(X)$, and since $\mu$ is a Radon measure, we have $\mu(K)>1 - \ep$ for some compact set $K\subset X$.

For every $\bar x\in K^n$, there exists an open neighborhood $V_{\bar x}\subset X^n$ of $\bar x$ and an integer 
$N_{\bar x}\in \N$ such that $U_m\cap V_{\bar x}=\es$ for $m>N_{\bar x}$. The family 
$\set{V_{\bar x}: \bar x\in K^n}$ is an open cover of the compact set $K^n$. Consequently, 
$K^n\subset \bigcup_{\bar x\in A} V_{\bar x}$ for some finite set $A\subset K^n$.
We set $V=\bigcup_{\bar x\in A} V_{\bar x}$ and $N=\max\set{N_{\bar x}:\bar x\in A}$.
Then $K^n\subset V$ and $U_m\cap V=\es$ for $m>N$.
Wallace's theorem \seee\cite[Theorem~3.2.10]{EngelkingBookRu} implies the existence of an open set 
$U\subset X$ such that
\[
K\subset U,\ \ K^n\subset U^n \subset V.
\]
The set $W(U,1-\ep)$ is a neighborhood of the measure $\mu$, because $\mu(K)>1 - \ep$.
Since the sequence $(W_m)_m$ accumulates to the measure $\mu$, 
we have $W(U,1-\ep)\cap W_m\neq\es$ for some $m>N$. Let $\nu\in W(U,1-\ep)\cap W_m$. 
Then $\nu(U)>1-\ep$ and $\nu(U_{m,i})>\ep$ for all $i=1,2,\dots,n$.
Consequently, $U\cap U_{m,i}\neq\es$ for all $i=1,2,\dots,n$, whence $U^n\cap U_m\neq\es$. 
This contradicts the relations $U_m\cap V=\es$ and $U^n\subset V$.
\end{proof}

\begin{proposition}
\label{f:defs:circled}
Any convex subset of a TVS is locally path-connected.
\end{proposition}

\begin{proof}
Let $C$ be a convex subset of a TVS $L$. A subset $V\subset L$ is said to be \term{circled} 
\see\cite{Schaefer1966ru} if $v\in V$ and $\la\in [-1,1]$ imply $\la v\in V$.
In TVS, open circled neighborhoods of zero form a local base at zero \seee\cite[1.2]{Schaefer1966ru}.
Let $C$ be a convex subset of TVS $L$, and let $\cN$ be a base of open circled neighborhoods of zero in $L$.
Take $v\in C$.
If $V\in\cN$ and $u\in M=(v+V)\cap C$, then $\la v+(1-\la)u\in M$ for $\la\in[0,1]$. Therefore, $M$ 
is path-connected.
The family $(v+V)\cap C$ for $V\in\cN$ forms a base of open path-connected neighborhoods of $v$ in~$C$.
\end{proof}

%%% end file('tex/sec/main.tex') %%%

\smallskip
The author thanks V.~I.~Bogachev for attention to the work and helpful discussions.
The author thanks the reviewers for carefully reading the paper and many corrections and comments, which 
improved the presentation.

\end{fulltext}

%\bibliographystyle{amsbibs}
%\bibliographystyle{amsbib}
%\bibliographystyle{plain}
%\bibliography{in/lit}

\def\elink#1{#1}
\def\bookvol#1{#1}
\def\arxiv#1{#1}
% +
%\input betap-amsbibs.bbl
%\input betap-amsbib.bbl
%%% begin file('tex/lit.bbl') %%%

%%% end file('tex/lit.bbl') %%%

\end{document}